\documentclass[11 pt]{amsart}
\usepackage{fancyhdr}
\usepackage[margin=1.15in]{geometry}
\usepackage{amsthm}
\usepackage{amsmath}
\usepackage{amssymb}
\usepackage{enumerate}
\usepackage{graphicx}
\usepackage{cancel}
\usepackage{hyperref}

\usepackage{comment}
\usepackage{enumitem}
\usepackage{mathtools}
\usepackage{tikz}
\usetikzlibrary{matrix,arrows,decorations.pathmorphing}
\usepackage{mathabx}
\usetikzlibrary{matrix}
\usepackage{scrextend}
\usepackage{tkz-graph}
\usetikzlibrary{patterns}
\usepackage{tikz-cd}

\newtheorem {theorem}{Theorem}[section]
\newtheorem {lemma}[theorem]{Lemma}

\newtheorem {corollary}[theorem]{Corollary}

\newtheorem {definition}[theorem]{Definition}
\theoremstyle{remark}
\newtheorem {remark}[theorem]{Remark}
\newtheorem {example}[theorem]{Example}

\DeclareMathOperator{\coker}{coker}
\DeclareMathOperator{\im}{im}

\newcommand\Z{\mathbb{Z}}

\newcommand\Q{\mathbb{Q}}

\def\H{\mathcal{H}}

\newcommand \bJ {\bar{J}}
\newcommand \bh {\mkern3mu \overline{\mkern-3mu H \mkern-1mu} \mkern1mu}
\newcommand \bH {\overline{\H}}

\newcommand\Ta{\mathbb{T}_\alpha}
\newcommand\Tb{\mathbb{T}_\beta}

\def\Sym{\mathrm{Sym}}

\def\del {{\partial}}

\def\spinc {{\operatorname{spin^c}}}
\def\Spinc {{\operatorname{Spin^c}}}

\def\fin\qedhere
\def\pr {{\text{pr}}}

\DeclareMathOperator{\Hom}{Hom}

\def\from {{\leftarrow}}

\def\s{\mathfrak s}

\def\Tower{\mathcal{T}^+}
\def\Ring {\mathcal R}

\newcommand\alphas{\boldsymbol\alpha}
\newcommand\betas{\boldsymbol\beta}

\def\bs{\bar{\mathfrak{\s}}}

\def\gr{\mathrm{gr}}

\def\ff {{\mathbb{F}}}

\def\ker {{\operatorname{ker}}}

\def\fin\qedhere
\def\from {{\leftarrow}}

\def\ccdot {\! \cdot \!}
\def\pin{\operatorname{Pin}(2)}

\def\Cone{\mathit{Cone}}

\newcommand{\bunderline}[1]{\underline{#1\mkern-2mu}\mkern2mu }

\def\du {\bar{d}}
\def\dl {\bunderline{d}}

\def\He{\mathbb{H}}

\def\CF {\mathit{CF}}
\def\HF {\mathit{HF}}
\newcommand\HFhat{\widehat{\HF}}
\newcommand\CFhat{\widehat{\CF}}
\newcommand\HFp {\HF^+}
\newcommand \CFp {\CF^+}
\newcommand \CFm {\CF^-}
\newcommand \HFm {\HF^-}
\newcommand \CFinf {\CF^{\infty}}
\newcommand \HFinf {\HF^{\infty}}
\newcommand \CFo {\CF^{\circ}}
\newcommand \HFo {\HF^{\circ}}

\def\HIm{\mathit{HI}^-}

\def\CFI {\mathit{CFI}}
\def\HFI {\mathit{HFI}}

\newcommand\HFIp {\HFI^+}

\newcommand \CFIm {\CFI^-}
\newcommand \HFIm {\HFI^-}

\newcommand \HFIo {\HFI^{\circ}}

\def\inv{\iota}

\def\Inv{\mathfrak{I}}

\newcommand{\co}{\nobreak\mskip2mu\mathpunct{}\nonscript
  \mkern-\thinmuskip{:}\penalty300\mskip6muplus1mu\relax}

\def\swf{\mathit{SWF}}

\def\rp{\mathbb{RP}}

\def\Vert{\operatorname{Vert}}
\def\Char{\operatorname{Char}}
\def\cQ{\mathcal{Q}}

\def\tdelta{\tilde\delta}

\input xy
\xyoption{all}

\begin{document}

\title{Involutive Heegaard Floer homology and plumbed three-manifolds}

\author[Irving Dai]{Irving Dai}
\author[Ciprian Manolescu]{Ciprian Manolescu}
\thanks{ID was partially supported by NSF grant DGE-1148900.}
\thanks {CM was partially supported by NSF grant DMS-1402914.}

\address{Department of Mathematics, Princeton University, Princeton, NJ 08540}
\email{idai@math.princeton.edu}

\address {Department of Mathematics, UCLA, 
Los Angeles, CA 90095}
\email {cm@math.ucla.edu}

\begin{abstract}
We compute the involutive Heegaard Floer homology of the family of three-manifolds obtained by plumbings along almost-rational graphs. (This includes all Seifert fibered homology spheres.) We also study the involutive Heegaard Floer homology of connected sums of such three-manifolds, and explicitly determine the involutive correction terms in the case that all of the summands have the same orientation. Using these calculations, we give a new proof of the existence of an infinite-rank subgroup in the three-dimensional homology cobordism group. 
\end{abstract}



\maketitle

\section{Introduction}
\label{sec:introduction}
In \cite{HMinvolutive}, Hendricks and the second author defined an invariant of three-manifolds called involutive Heegaard Floer homology. This is a variation of the well-known Heegaard Floer homology of Ozsv\'ath and Szab\'o \cite{HolDisk, HolDiskTwo}, taking into account the conjugation action $\inv$ on the Heegaard Floer complex. Specifically, let $\CFo(Y, \s)$ denote any of the four flavors of the Heegaard Floer complex (with $\circ = \widehat{\phantom{a}}$, $+$, $-$, or $\infty$), for a three-manifold $Y$ with a self-conjugate $\spinc$ structure $\s$. The Heegaard Floer complex is   constructed starting from a Heegaard diagram for $Y$. The conjugation $\iota$ is induced by reversing the orientation of the Heegaard surface and swapping the alpha and beta curves, and observing that the result is also a Heegaard diagram for $Y$, which can be related to the original one by a sequence of Heegaard moves. The corresponding involutive Heegaard Floer homology $\HFIo(Y, \s)$ is the homology of the mapping cone
\begin{equation}
\label{eq:CFI}
\CFo(Y, \s) \xrightarrow{\phantom{o} Q (1+\inv) \phantom{o}} Q \ccdot \CFo(Y, \s) [-1]. 
\end{equation}
This is a module over the ring $\Ring = \ff[Q, U]/(Q^2)$, where $\ff$ denotes the field of two elements.

The construction of involutive Heegaard Floer homology was inspired by developments in gauge theory \cite{Triangulations, Lin, Stoffregen}. Indeed, in the case when $Y$ is a rational homology sphere, $\HFIp(Y, \s)$ is conjectured to be isomorphic to the $\Z_4$-equivariant Seiberg-Witten Floer homology of the pair $(Y, \s)$. Just as Heegaard Floer homology is more amenable to computations than Seiberg-Witten theory, involutive Heegaard Floer homology should be easier to calculate than its gauge-theoretic counterpart. So far, $\HFIo$ was computed in \cite{HMinvolutive} for large surgeries on $L$-space and thin knots (including many hyperbolic examples); see also \cite{BorodzikHom} for related applications. Moreover, a formula for $\HFIo$ of connected sums was established in \cite{HMZ}.

The purpose of this paper is to compute $\HFIo$ for the class of rational homology spheres introduced by N{\'e}methi in \cite{NemethiOS}, namely those associated to plumbings along AR (almost-rational) graphs. For the sake of the exposition, we will focus on the minus version, $\HFIm$.

If $G$ is a weighted graph, let us denote by $Y(G)$ the boundary of the corresponding plumbing of two-spheres. In \cite{Plumbed},  Ozsv\'ath and Szab\'o calculated the  Heegaard Floer homology of $Y(G)$ when $G$ is a negative definite graph with at most one bad vertex. (In particular, if a rational homology sphere $Y$ is Seifert fibered over a base orbifold with underlying space $S^2$, then $Y$ can be expressed as $Y(G)$ for such a graph $G$.) In \cite{NemethiOS}, N{\'e}methi defined AR graphs as those obtained from a rational graph by increasing the decoration of at most one vertex. We will call the resulting manifolds {\em AR plumbed three-manifolds}, for simplicity. They include those considered by Ozsv\'ath and Szab\'o in \cite{Plumbed}, as well as the links of rational and weakly elliptic singularities. Further, extending the work of \cite{Plumbed}, N{\'e}methi computed the Heegaard Floer homology of AR plumbed three-manifolds. Specifically, to each plumbing graph $G$ and characteristic element $k$ one can associate a decorated infinite tree $R_k$, called a graded root. To $R_k$ one can associate, in a combinatorial fashion, a lattice homology group $\He^-(G, k)=\He^-(R_k)$. N{\'e}methi's result, phrased in terms of the minus rather than the plus flavor, is that
$$ \HFm(Y(G), [k]) \cong \He^-(R_k)[\sigma+2].$$
Here, $[k]$ denotes the $\spinc$ stucture on $Y(G)$ associated to $k$, and $\sigma + 2$ denotes a  grading shift,\footnote{In this paper, by $[a]$ we will denote a grading shift by $a$, so that an element in degree $0$ in a module $M$ becomes an element in degree $-a$ in the shifted module $M[a].$ This is the standard convention, but opposite to the one in \cite{NemethiOS}.} where $\sigma=-({|G| + k^2})/{4}$.

To understand the involutive Heegaard Floer homology, we focus on self-conjugate $[k]$. In that case, there is a natural involution $J_0$ on $\He^-(R_k)$ (and hence on its grading-shifted counterpart), coming from reflecting the graded root along its vertical axis. 

\begin{theorem}
\label{thm:main}
Let $Y=Y(G)$ be the plumbed three-manifold associated to an AR graph $G$, and oriented as the boundary of the plumbing. Let $\s=[k]$ be a self-conjugate $\spinc$ structure on $Y$. Let also $R_k$ be the graded root associated to $(G, k)$, and let $J_0$ be the reflection involution on the lattice homology $\He^-(R_k)[\sigma + 2]$. 

Then, we have an isomorphism of graded $\ff[U]$-modules
\[
\HFIm(Y, \s) \cong \ker (1 + J_0)[-1] \oplus \coker (1 + J_0).
\]
Under this isomorphism, the action of $Q$ on $\HFIm(Y, \s)$ is given by the quotient map
\[
\ker (1 + J_0) \rightarrow \ker (1 + J_0)/\im (1 + J_0) \subseteq \coker (1 + J_0).
\]
\end{theorem}

The main sources of applications for involutive Heegaard Floer homology are the involutive correction terms $\dl(Y, \s)$ and $\du(Y, \s)$, the analogues of the Ozsv\'ath-Szab\'o correction term $d(Y,\s)$ defined in \cite{AbsGraded}. In \cite{HMinvolutive}, it is shown that $\dl$ and $\du$ can be used to constrain the intersection forms of spin four-manifolds with boundary, and the existence of homology cobordisms between three-manifolds. 

For the class of plumbed three-manifolds studied in this paper, we have the following result.
\begin{theorem}
\label{thm:ds}
Let $Y$ and $\s$ be as in Theorem~\ref{thm:main}. Then, the involutive Heegaard Floer correction terms are given by 
$$\dl(Y, \s) = - 2\bar{\mu}(Y, \s), \ \ \du(Y, \s) = d(Y, \s),$$ where $d(Y, \s)$ is the Ozsv\'ath-Szab\'o $d$-invariant and $\bar{\mu}(Y, \s)$ is the Neumann-Siebenmann invariant from \cite{Neu}, \cite{Sieb}.
\end{theorem}

Theorems~\ref{thm:main} and \ref{thm:ds} should be compared with the corresponding results for $\pin$-monopole Floer homology, obtained by the first author in \cite{Dai}. For the smaller class of  rational homology spheres Seifert fibered over a base orbifold with underlying space $S^2$, the $\pin$-equivariant Seiberg-Witten Floer homology was computed by Stoffregen in \cite{Stoffregen}.

The proof of Theorem~\ref{thm:main} is in two steps. First, we identify the action of $J_0$ on $\He^-(R_k)$ (or, more precisely, $\He^-(R_k)[\sigma+2])$ with the conjugation involution on $\HFm(Y, \s)$, using an argument from \cite{Dai}. Second, we prove that, in the case at hand, the action of $J_0$ on $\He^-(R_k)$ determines (up to chain homotopy) the underlying action at the chain level, on any free chain complex that computes $\He^-(R_k)$. This allows us not only to compute the involutive Heegaard Floer homology, but also to characterize the pair $(\CFm(Y, \s), \inv)$ up to the notion of equivalence considered in \cite[Definition 8.3]{HMZ}. 

The equivalence class of $(\CFm, \inv)$ is the input needed for the connected sum formula in involutive Heegaard Floer homology proved in \cite{HMZ}. Thus, we can use that formula to calculate $\HFIm$ for connected sums of AR plumbed three-manifolds. In this paper we will focus on explicitly calculating the involutive correction terms for such manifolds. To compute the correction terms we only need to understand the pair $(\CFm, \inv)$ up to a weaker equivalence relation, called local equivalence; cf. \cite[Definition 8.5]{HMZ}. 

As we prove in Sections~\ref{sec:monotone} and \ref{sec:dssums}, for an AR plumbed manifold $Y$ with self-conjugate $\spinc$ structure $\s$, the local equivalence class of $(\CFm(Y, \s), \inv)$ is characterized by $2n$ quantities 
$$ d_1(Y, \s) > d_2(Y, \s) > \dots > d_n(Y, \s), $$
$$ \bar \mu_1(Y, \s) > \bar \mu_2(Y, \s) > \dots > \bar \mu_n(Y, \s),$$
where $d_1=d$ is the Ozsv\'ath-Szab\'o correction term and $\bar \mu_n = \bar \mu$ is the Neumann-Siebenmann invariant. The values of $d_i, \bar \mu_i$ can be read from the graded root $R_k$ associated to $(Y, \s)$. Specifically, these values describe a simplified graded root, called {\em monotone}, which is obtained from $R_k$ by deleting some branches according to a certain algorithm. 

We refer to the integer $n$ above (the number of quantities $d_i$) as the {\em local complexity} of $(Y, \s)$. In particular, having local complexity one is the same as the {\em projective type} condition introduced by Stoffregen in \cite[Section 5.2]{Stoffregen}.

It will be convenient to also define
\[
2\tdelta_i(Y, \s) = d_i(Y, \s) + 2\bar \mu_i(Y, \s), \ \ \ 2\tdelta(Y, \s) = d(Y, \s) + 2\bar \mu(Y, \s).
\]

\begin{theorem}\label{thm:dssums}
Let $Y_1, \dots, Y_k$ be AR plumbed three-manifolds, oriented as the boundaries of those plumbings, and equipped with self-conjugate $\spinc$ structures $\s_1, \dots, \s_k$. Suppose that $(Y_i, \s_i)$ has local complexity $n_i$. 

$(a)$ We have
\[
\du(Y_1 \# \dots \# Y_k, \s_1 \# \dots \# \s_k) = d(Y_1 \# \dots \# Y_k, \s_1 \# \dots \# \s_k) = \sum_{i=1}^k d(Y_i, \s_i).
\]

$(b)$ For each tuple of integers $(s_1, \ldots, s_k)$ with $1 \leq s_i \leq n_i$, define
\[
N(s_1, \ldots, s_k) = \left( \sum_{i = 1}^k d_{s_i}(Y_i, \s_i) \right)- \max_i \left( 2\tdelta_{s_i}(Y_i, \s_i) \right).
\]
Then
\[
\dl(Y_1 \# \dots \# Y_k, \s_1 \# \dots \# \s_k) = \max_{(s_1, \ldots, s_k)} N(s_1, \ldots, s_k).
\]
where the maximum is taken over all tuples $(s_1, \ldots, s_k)$ with $1 \leq s_i \leq n_i$.
\end{theorem}

Part (b) of Theorem~\ref{thm:dssums} has the following consequence.
\begin{corollary}
\label{cor:dssums}
Let $(Y_i, \s_i)$ be as in Theorem~\ref{thm:dssums}. Without loss of generality, assume that 
$$\tdelta(Y_1, \s_1) \leq \dots \leq \tdelta(Y_k, \s_k).$$ Then:
\begin{equation}
\label{eq:dlsumsinequality}
\dl(Y_1 \# \dots \# Y_k, \s_1 \# \dots \# \s_k) \leq\Bigl( \sum_{i=1}^{k-1} d(Y_i, \s_i)\Bigr) - 2\bar \mu(Y_k, \s_k) = 2 \Bigl( \sum_{i=1}^{k-1} \tdelta_i -  \sum_{i=1}^k \bar \mu(Y_i, \s_i) \Bigr).
\end{equation}

Moreover, if $(Y_i, \s_i)$ are all of projective type, then we have equality in \eqref{eq:dlsumsinequality}.
\end{corollary}

In \cite[Theorem 1.4]{Stoffregen2}, Stoffregen calculated the homology cobordism invariants $\alpha, \beta, \gamma$ (coming from $\pin$-equivariant Seiberg-Witten Floer homology, cf. \cite{Triangulations}) in the case of connected sums of Seifert fibered integral homology spheres of projective type. Corollary~\ref{cor:dssums} implies that $\dl = 2\beta$ or $2\beta-2$ for those manifolds.

In a different direction, we can specialize Theorem~\ref{thm:dssums} to the case of self-connected sums $\#_k(Y, \s)$. In that case, the expression for $\dl$ simplifies significantly. Indeed, the following corollary shows that $\dl$ satisfies a curious ``stabilization" property for self-connected sums of AR manifolds. More precisely, it turns out that for $k$ sufficiently large, $\dl(\#_k (Y, \s))$ is a linear function of $k$:

\begin{corollary}
\label{cor:stab}
Let $Y$ be an AR plumbed three-manifold and let $\s$ be a self-conjugate $\spinc$ structure on $Y$. Then
\[
\dl(\#_k(Y, \s)) = \max_i \left( k \cdot d_i(Y, \s) - 2\tdelta_i(Y, \s) \right).
\]
Moreover, for all $k$ sufficiently large, this is equal to $k \cdot d(Y, s) - 2\tdelta_1(Y, s)$.
\end{corollary}

Corollary~\ref{cor:stab} should be compared with Theorem 1.3 of \cite{Stoffregen2}, which shows that for $\alpha$, $\beta$, and $\gamma$, the difference $\beta(\#_k(Y, \s)) - k\cdot d(Y, \s)$ (for example) is a bounded function of $k$. Here, in the case of AR manifolds, we have the stronger result that for $\du$ and $\dl$, the corresponding expressions are eventually constant (namely, identically zero for $\du$, and  $-2\tdelta_1(Y, \s)$ for $\dl$).

We can also study $\HFIm$ for connected sums of AR plumbed three-manifolds, when some of these manifolds are equipped with the opposite orientation. Recall that in \cite{HMZ} such a connected sum (of large surgeries on torus knots) was used to give an example of a $\Z_2$-homology sphere with $\dl, d$ and $\du$ all different. Using the methods of this paper, we can now find an integral homology sphere with the same property.

\begin{theorem}
\label{thm:AllDifferent}
Consider the integral homology $3$-sphere
$$ Y = S^3_{-1}(T_{2,7}) \# S^3_{-1}(T_{2,7}) \# S^3_{1}(-T_{2,11})
= \Sigma(2,7,15) \# \Sigma(2,7,15) \# -\Sigma(2,11,23).$$
Then, we have
$$ \dl(Y) = -2, \ \ d(Y) = 0,\ \ \du(Y)=2.$$
Consequently, $Y$ is not homology cobordant to any AR plumbed three-manifold.
\end{theorem}

Finally, we give a new proof of the following result:
\begin{theorem}[cf. Fintushel-Stern \cite{FSinstanton}, Furuta \cite{FurutaHom}]
\label{thm:infgen}
The homology cobordism group $\Theta^3_{\Z}$ is infinitely generated.
\end{theorem}

The original proofs of Theorem~\ref{thm:infgen} used Yang-Mills theory. Recently, Stoffregen \cite{Stoffregen2} gave another proof, which is based on $\pin$-equivariant Seiberg-Witten Floer homology. Our proof is modelled on that of Stoffregen, but uses involutive Heegaard Floer homology---thus, replacing gauge theory with symplectic geometry. Specifically, we show that the Brieskorn spheres
$$  \Sigma(p, 2p-1, 2p+1),\ p\geq 3, \ p \text{ odd}$$
are linearly independent in $\Theta^3_{\Z}$, forming a $\Z^{\infty}$ subgroup. (The same manifolds were considered by Stoffregen in \cite{Stoffregen2}.)

\medskip
\noindent {\bf Organization of the paper.} Section~\ref{sec:background} reviews some background material on involutive Heegaard Floer homology, plumbings, graded roots, and lattice homology. In Section~\ref{sec:J}, we identify the action of the conjugation involution on the Heegaard Floer homology of AR plumbed three-manifolds. In Section~\ref{sec:standard} we describe a free complex that computes the lattice homology $\He^-(R_k)$, which we call the standard complex associated to the graded root; moreover, we prove that a complex of this form (with its involution) is uniquely determined by $(R_k, J_0)$ up to equivalence. This key fact is used in Section~\ref{sec:hfiAR} to prove Theorems \ref{thm:main} and \ref{thm:ds}. In Section~\ref{sec:monotone} we describe the local equivalence class of a graded root, in terms of an associated monotone root. In Section~\ref{sec:tensor} we study the tensor products of standard complexes associated to monotone graded roots. These results are then put to use in Section~\ref{sec:app}, where we prove Theorems~\ref{thm:dssums}, \ref{thm:AllDifferent} and \ref{thm:infgen}, as well as Corollaries~\ref{cor:dssums} and \ref{cor:stab}.

\medskip
\noindent {\bf Acknowledgements.} We thank Tye Lidman, Matthew Stoffregen, Zolt\'an Szab\'o and Ian Zemke for helpful conversations and suggestions.

\section{Background}
\label{sec:background}

\subsection{Involutive Heegaard Floer homology}
\label{sec:hfi}
We assume that the reader is familiar with Heegaard Floer homology, as in \cite{HolDisk, HolDiskTwo, HolDiskFour, AbsGraded}. Throughout the paper we will work with coefficients in $\ff=\Z/2$. Furthermore, we will only consider rational homology spheres $Y$, equipped with self-conjugate $\spinc$ structures $\s$. In this situation, the Heegaard Floer groups
$$ \HFhat(Y, \s), \HFp(Y, \s), \HFm(Y, \s), \HFinf(Y, \s)$$
have absolute $\Q$-gradings, and are modules over the ring $\ff[U]$. We focus on the minus version, which is the homology of a complex $\CFm=\CFm(Y, \s)$, consisting of free $\ff[U]$-modules. The other three Floer complexes can be obtained from $\CFm$:
\begin{equation}
\label{eq:FromMinus}
 \CFhat = \CFm/(U=0), \  \CFinf = U^{-1}\CFm, \ \CFp = \CFinf/\CFm.
 \end{equation}

The construction of $\CFo(Y, \s)$ relies on fixing a particular Heegaard pair $\H = (H, J)$ for $Y$, consisting of a pointed Heegaard diagram $H = (\Sigma, \alphas, \betas, z)$, with $\Sigma$ embedded in $Y$, together with a family of almost-complex structures $J$ on $\Sym^g(\Sigma)$. We use the notation $\CFo(\H, \s)$ to emphasize the dependence of $\CFo(Y, \s)$ on $\H$. If $\H$ and $\H'$ are two Heegaard pairs for $Y$ with the same basepoint $z$, then by work of Juh\'asz and Thurston \cite{Naturality}, we obtain a preferred isomorphism 
$$ \Psi(\H, \H'): \HFo(\H, \s) \rightarrow \HFo(\H', \s). $$
More precisely, choosing any sequence of moves relating $\H$ and $\H'$ yields a chain homotopy equivalence 
$$ \Phi(\H, \H'): \CFo(\H, \s) \rightarrow \CFo(\H', \s) $$
inducing the isomorphism $\Psi$ above. The map $\Phi$ is itself unique up to chain homotopy in the sense that choosing any other sequence of moves yields a map chain-homotopic to it; cf. \cite[Proposition 2.3]{HMinvolutive}. This justifies the use of the notation $\CFo(Y, \s)$, rather than $\CFo(\H, \s)$.

We now define a grading-preserving homotopy involution on $\CFo(Y, \s)$. Given a Heegaard pair $\H$ for $Y$, define the conjugate Heegaard pair $\bH$ to be the conjugate Heegaard diagram
$$ \bh = (-\Sigma, \betas, \alphas, z), $$
together with the conjugate family of almost-complex structures $\bJ$ on $\Sym^g(-\Sigma)$. Then intersection points in $\Ta \cap \Tb$ for $H$ are in one-to-one correspondence with those for $\bh$, and $J$-holomorphic disks with boundary on $(\Ta, \Tb)$ are in one-to-one correspondence with $\bJ$-holomorphic disks with boundary on $(\Tb, \Ta)$. Thus, as observed in \cite[Theorem 2.4]{HolDiskTwo}, we obtain a canonical isomorphism
$$ \eta: \CFo(\H, \s) \rightarrow \CFo(\bH, \s), $$
where we have used the fact that in our case $\s = \bs$. Since $\H$ and $\bH$ are Heegaard pairs for the same three-manifold, we may further define the composition
$$ \inv = \Phi(\H', \H) \circ \eta: \CFo(\H, \s) \rightarrow \CFo(\H, \s). $$
In \cite[Section 2.2]{HMinvolutive} it is shown that $\inv$ is a well-defined map on $\CFo(Y, \s)$ (up to the notion of equivalence considered below) and that $\inv^2$ is chain-homotopic to the identity.

The involutive Heegaard Floer complex $\CFIm(Y, \s)$ is then in \cite{HMinvolutive} defined to be the mapping cone of $1 + \inv$ on $\CFm$, with a new variable $Q$ that marks the target (as opposed to the domain) of the cone; cf. the formula \eqref{eq:CFI}. By taking homology we obtain the invariant $\HFIm(Y, \s)$. The other three flavors of involutive Heegaard Floer homology are constructed similarly.

We can formalize the algebra underlying $\CFIm(Y, \s)$ as in \cite[Section 8]{HMZ}. The following is Definition 8.1 in \cite{HMZ}, slightly modified to allow for $\Q$-gradings:
\begin{definition}
\label{def:icx}
An {\em $\inv$-complex} is a pair $(C, \inv)$, consisting of
\begin{itemize}
\item a $\Q$-graded, finitely generated, free chain complex $C$ over the ring $\ff[U]$, where $\operatorname{deg}(U)=-2$. Moreover, we ask that there is some $\tau \in \Q$ such that the complex $C$ is supported in degrees differing from $\tau$ by integers. We also require that there is a relatively graded isomorphism
\begin{equation}
\label{eq:Utail}
U^{-1}H_*(C) \cong \ff[U, U^{-1}],
\end{equation}
and that $U^{-1}H_*(C)$ is supported in degrees differing from $\tau$ by even integers;
\item a grading-preserving chain homomorphism $\inv \co C \to C$, such that $\inv^2$ is chain homotopic to the identity.
\end{itemize}
\end{definition}

An example of an $\inv$-complex is $\CFm(Y, \s)$, equipped with the conjugation involution. Further, when $Y$ is an integral homology sphere, this is an $\inv$-complex where we can take $\tau=0$.

Let us also recall Definition 8.3 in \cite{HMZ}:
\begin{definition}
\label{def:E}
Two $\inv$-complexes $(C, \inv)$ and $(C', \inv')$ are called {\em equivalent} if there exist chain homotopy equivalences
$$ F \co C \to C', \ \ G \co C' \to C$$
that are homotopy inverses to each other, and such that 
$$F \circ \inv \simeq \inv' \circ F,  \ \ \ G \circ \inv' \simeq \inv \circ G,$$
where $\simeq$ denotes $\ff[U]$-equivariant chain homotopy.
\end{definition}

To each $\inv$-complex $(C, \inv)$, we can associate an {\em involutive complex}
\begin{equation}
\label{eq:invcx}
I^-(C, \inv) := \Cone\bigl(C_* \xrightarrow{\phantom{o} Q (1+\inv) \phantom{o}} Q \ccdot C_* [-1]\bigr),
 \end{equation}
For example, $I^-(\CFm, \inv) = \CFIm$. We refer to the homology of the mapping cone as the \textit{involutive homology} of $(C, \inv)$ and denote it by $\HIm(C, \inv)$.

The proof of the following lemma is based on a simple filtration argument; compare \cite[proof of Proposition 2.8]{HMinvolutive}.
\begin{lemma}
\label{lem:EquivCones}
An equivalence of $\inv$-complexes induces a quasi-isomorphism between the respective involutive complexes \eqref{eq:invcx}. 
\end{lemma}

In this paper, we will describe the equivalence class of the pair $(\CFm(Y, \s), \inv)$ for certain three-manifolds. By the above lemma, the equivalence class determines the involutive Heegaard Floer homology $\HFIm(Y, \s)$, up to isomorphism. Furthermore, from $\CFm$ we can derive the other Heegaard Floer complexes using \eqref{eq:FromMinus}, and the map $\inv$ descends to similar maps on those complexes. From here, we can get the corresponding flavors of involutive Heegaard Floer homology, again up to isomorphism.

The equivalence class of $(\CFm, \inv)$ also appears in the following connected sum formula, which follows from \cite[Theorem 1.1]{HMZ}.
\begin{theorem}[\cite{HMZ}]
\label{thm:ConnSum}
Suppose $Y_1$ and $Y_2$ are rational homology spheres equipped with self-conjugate $\spinc$ structures $\s_1$ and $\s_2$. Let $\inv_1$, $\inv_2$ and $\inv$ denote the conjugation involutions on the Floer complexes $\CFm(Y_1, \s_1)$, $\CFm(Y_2,\s_2)$ and $\CFm(Y_1 \# Y_2, \s_1 \# \s_2)$. Then, the equivalence class of the $\inv$-complex $(\CFm(Y_1 \# Y_2), \inv)$ is the same as that of
$$ \bigl( \CFm(Y_1, \s_1) \otimes_{\ff[U]} \CFm(Y_2, \s_2) [-2], \ \inv_1 \otimes \inv_2 \bigr),$$
where $[-2]$ denotes a grading shift.
\end{theorem}

\begin{remark}
The grading shift is due to the fact that, with the usual conventions in Heegaard Floer theory, $\HFm(S^3)=\Z$ is supported in degree $-2$.  
\end{remark}

We now review the definition of the correction terms. In \cite{AbsGraded}, Ozsv\'ath and Szab\'o defined the $d$-invariant as the minimal degree of a nonzero element in the infinite tower of $\HFp$; or, equivalently, as the maximal degree of a nonzero element in the infinite tail of $\HFm$, plus two:
$$ d(Y, \s) =  \max \{r \mid \exists \ x \in \HFm_r(Y,\s),\forall \ n \geq 0, U^nx\neq 0\} +2. $$

In involutive Heegaard Floer homology, there are two analogous invariants
\[ \dl(Y,\s) = \max \{r \mid \exists \ x \in \HFIm_r(Y, \s), \forall \ n \geq 0, \ U^nx\neq 0 \ \text{and} \ U^nx \notin \operatorname{Im}(Q)\} + 1 \]
and
\[ \du(Y,\s) = \max \{r \mid \exists \ x \in \HFIm_r(Y,\s),\forall \ n \geq 0, U^nx\neq 0; \exists \ m\geq 0 \ \operatorname{s.t.} \ U^m x \in \operatorname{Im}(Q)\} +2.\]
See \cite[Section 5.1]{HMinvolutive} and \cite[Lemma 2.9]{HMZ}. 

(The shifts by $1$ and $2$ in these definitions are chosen so that $d=\dl=\du=0$ for $Y=S^3$.)

The above invariants satisfy
$$
\dl(Y,\s) \leq  d(Y,\s) \leq \du(Y,\s)$$
and
$$\du(Y, \s) \equiv  \dl(Y, \s) \equiv d(Y, \s) \pmod {2\Z}.$$
Moreover, they descend to maps
$$ d, \dl, \du: \Theta^3_{\Z} \to 2\Z,$$
where $\Theta^3_{\Z}$ denotes the homology cobordism group (consisting of integral homology $3$-spheres, up to the equivalence given by homology cobordisms). Of the three maps above, only $d$ is a homomorphism.

For the purpose of computing the involutive Floer correction terms, it will be useful for us to consider a weaker notion of equivalence than discussed above. Recall Definition 8.5 of \cite{HMZ}:
\begin{definition}
\label{def:localE}
Two $\inv$-complexes $(C, \inv)$ and $(C', \inv')$ are called {\em locally equivalent} if there exist (grading-preserving) homomorphisms
$$ F \co C \to C', \ \ G \co C' \to C$$
such that 
$$F \circ \inv \simeq \inv' \circ F,  \ \ \ G \circ \inv' \simeq \inv \circ G,$$
and $F$ and $G$ induce isomorphisms on homology after inverting the action of $U$.
\end{definition}

Definition \ref{def:localE} is modelled on the relation between Floer chain complexes induced by a homology cobordism. As observed in \cite[Lemma 8.7]{HMZ}, the relation of local equivalence respects taking tensor products, in the sense that if $(C, \inv)$ or $(C, \inv')$ is changed by a local equivalence, then their tensor product also changes by a local equivalence. One can thus define a group structure on the set of $\inv$-complexes modulo local equivalence, with multiplication given by the tensor product; cf. \cite[Section 8.3]{HMZ}. In \cite{HMZ} only $\inv$-complexes with $\tau=0$ were considered, and the resulting group was denoted $\Inv$. One can also form such a group from all $\inv$-complexes, which we denote by $\Inv_{\Q}$. Of course, $\Inv_{\Q}$ is just the direct sum of infinitely many copies of $\Inv$, one for each $[\tau] \in \Q/2\Z$.

Given an $\inv$-complex $(C, \inv)$, we define $\dl_{\Inv}$ and $\du_{\Inv}$ analogously to the involutive Floer correction terms $\dl$ and $\du$ above. The same argument that shows $\dl$ and $\du$ are homology cobordism invariants (cf. \cite[Proposition 5.4]{HMinvolutive}) proves that these depend only on the local equivalence class of $(C, \inv)$. Hence we have (non-homomorphisms)
$$ \dl_{\Inv}, \du_{\Inv}: \Inv_{\Q} \rightarrow \Q. $$
If we have an actual Heegaard Floer complex $(\CFm(Y, \s), \iota)$, we can identify it with its equivalence class in $\Inv_{\Q}$, in which case $\dl_{\Inv} = \dl$ and $\du_{\Inv} = \du$. Theorem~\ref{thm:ConnSum} states that under this identification, the Floer $\inv$-complex corresponding to the connected sum of $(Y_1, \s_1)$ and $(Y_2, \s_2)$ is mapped to the grading-shifted tensor product of their images in $\Inv_{\Q}$. 


\subsection{Plumbing graphs}
\label{sec:pg}
We review here some facts about plumbings, with an emphasis on the almost-rational graphs introduced in \cite{NemethiOS}. We refer to \cite{Neu}, \cite{Plumbed} and \cite{NemethiOS} for more details.

Let $G$ be a weighted graph, i.e., a graph equipped with a function $m: \Vert(G) \to \Z$, where $\Vert(G)$ denotes the set of vertices of $G$. Let $W=W(G)$ be the four-manifold with boundary obtained by attaching two-handles to $B^4$ (or, equivalently, plumbing together disk bundles over $S^2$) according to the graph $G$. We denote by $Y=Y(G)$ the boundary of $W$, with the induced orientation.

Each vertex $v \in \Vert(G)$ gives rise to a generator of $H_2(W; \Z)$ represented by the core of the corresponding two-handle. The intersection form on the integer lattice $L = H_2(W; \Z)$ is given by
$$ \langle v, w \rangle = \begin{cases}
m(v) & \text{if $v=w$},\\
1 & \text{if $v$ and $w$ are connected by an edge},\\
0 & \text{otherwise.}
\end{cases}$$

From now on we will assume that the graph $G$ is a tree and the intersection form on $L$ is negative definite; this is equivalent to $Y$ being a rational homology sphere that is realized as the link of a  normal surface singularity. 

For $x \in L$, we write $x \geq 0$ if $x$ is a linear combination of the generators $v \in \Vert(G)$ with non-negative coefficients. Further, we write $x > 0$ if $x \geq 0$ and $x \neq 0$. 

The dual lattice $L' = \Hom(L, \Z)$ may be identified with $H_2(W, Y; \Z)$, and the exact sequence
$$ 0 \to H_2(W; \Z) \to H_2(W, Y ; \Z) \to H_1(Y; \Z) \to 0$$
shows that we can view $L$ as a sublattice of $L'$. The intersection form on $L$ extends to a $\Q$-valued intersection form of $L'$.

Let $\Char(G)$ be the set of characteristic vectors for $G$, i.e., vectors $k \in L'$ such that 
$$k(x) \equiv \langle x, x \rangle \mod{2} $$
for all $x \in L$. There is a natural action of $L$ on $\Char(G)$ given by $x: k \mapsto k + 2x$, and we denote the orbit of any $k \in \Char(G)$ under this action by $[k]$. Characteristic vectors $k$ may be identified with $\spinc$ structures on $W$, and equivalence classes $[k]$ may be identified with $\spinc$ structures on $Y$. Self-conjugate $\spinc$ structures on $Y$ correspond to the orbits $[k]$ lying entirely in the sublattice $L$, that is, those which satisfy $[k] = [-k]$.

Fix $k \in \Char(G)$, and define $\chi_k: L \to \Z$ by
\begin{equation}
\label{eq:chik}
 \chi_k(x) = - (k(x) + \langle x, x \rangle)/2.
 \end{equation}

There is a {\em canonical characteristic element} $K \in \Char(G)$ characterized by $K(v) = -m(v)-2$ for all $v \in \Vert(G)$. If we view $Y$ as the link of a normal surface singularity $X$ with resolution $\tilde X$, then $K$ is the first Chern class of the canonical bundle of the complex structure on $\tilde X$. 

A normal surface singularity is called {\em rational} if its geometric genus is zero. By the work of Artin \cite{Artin1, Artin2}, the manifold $Y = Y(G)$ is the boundary of a rational normal surface singularity if and only if
$$ \chi_K(x) \geq 1 \text{ for any } x > 0, x \in L.$$
If $G$ satisfies this property, we will say that $G$ is a rational plumbing graph.

\begin{definition}[N{\'e}methi \cite{NemethiOS}]
Let $G$ be a weighted tree with a negative definite intersection form on $L$, as above. We say that $G$ is {\em AR (almost-rational)} if there exists a vertex $v_0$ of $G$ such that by replacing the value $m_0 = m(v_0)$ with some $m'_0 \leq m_0$ we obtain a rational plumbing graph $G'$.

In this situation, we refer to $Y=Y(G)$ as an {\em AR plumbed three-manifold}.
\end{definition}

The class of AR graphs includes:
\begin{itemize}
\item rational graphs;
\item elliptic graphs (corresponding to normal surface singularities of geometric genus equal to one);
\item negative-definite, star-shaped graphs;
\item negative-definite plumbing trees with at most one bad vertex, that is, a vertex $v$ such that $m(v) > -d(v)$, where $d(v)$ is the degree of $v$. This is the class of plumbings considered by Ozsv\'ath and Szab\'o in \cite{Plumbed}.
\end{itemize}

In particular, if $Y$ is a rational homology sphere $Y$ that is Seifert fibered over an orbifold with underlying space $S^2$, then $Y$ is an AR plumbed three-manifold coming from a star-shaped graph $G$. Indeed, suppose that $Y$ is the Seifert bundle with negative orbifold Euler number $e$, and Seifert invariants
$$ (a_1, b_1), \dots, (a_n, b_n),$$
with $0 < b_i < a_i$, $\gcd(a_i, b_i) = 0$. Write $a_i/b_i$ as a continued fraction
$$a_i/b_i = k_1^i - \cfrac{1}{ k_2^i - \cfrac{1} {\ddots \, - \cfrac{1}{ k_{r_i}^i}}}$$ Then, $G$ is the star-shaped graph with $r$ arms, such that the decoration of the central vertex is
$$ e_0=e- \sum_{i=1}^n \frac{b_i}{a_i}$$
and along the $i^{\operatorname{th}}$ arm we see the decorations $-k_1^i, -k_2^i, \dots, -k_{r_i}^i$, in that order starting from the central vertex.

When
\begin{equation}
\label{eq:e}
 a_1 a_2 \dots a_n e = -1,
 \end{equation}
then $Y$ is an integral homology sphere, denoted $\Sigma(a_1, \dots, a_n)$. The values of $b_i$ are uniquely determined by the values of $a_i$, the fact that $e_0 \in \Z$, and the condition \eqref{eq:e}. Moreover, any Seifert fibered integral homology sphere is of the form $\pm \Sigma(a_1, \dots, a_n)$ for some $a_i$.

\begin{remark}
All Seifert fibered rational homology spheres have base orbifold with underlying space $S^2$ or $\rp^2$. (Further, if they are integer homology spheres, this has to be $S^2$.) When the underlying space is $\rp^2$, then the Seifert fibered rational homology sphere is an L-space, by \cite[Proposition 18]{BoyerGordonWatson}. We do not study this case in our paper, since the involutive Heegaard Floer homology of L-spaces is easily determined by the ordinary Heegaard Floer homology; cf. \cite[Corollary 4.8]{HMinvolutive}.
\end{remark}

\subsection{Graded roots and lattice homology}
\label{sec:lattice}
The Heegaard Floer homology of an AR plumbed three-manifold $Y(G)$ can be described in terms of a combinatorial object, called a {\em graded root}, which is associated to the weighted graph $G$. Graded roots were introduced in \cite{NemethiOS}, where the focus was on describing the groups $\HFp(-Y(G))$. In this paper we will concentrate on the minus version, $\HFm(Y(G))$. By the duality isomorphism from \cite[Proposition 7.11]{HolDiskFour}, we have
\begin{equation}
\label{eq:dualHF}
\HF_*^{-}(Y, \s) \cong \HF_+^{-*-2}(-Y, \s).
\end{equation}
Thus, we can consider the same graded roots as in \cite{NemethiOS}, except here we find it more convenient to reverse the grading, i.e., reflect them vertically. While the graded roots in \cite{NemethiOS} have an infinite upward stem and finite roots below, ours will have an infinite downward stem and a finite tree opening upward.

Specifically, with our conventions, a {\em graded root} is an infinite tree $R$ equipped with a grading function $\chi: \Vert(R) \to \Q$ such that
\begin{itemize}
\item $\chi(u) - \chi(v) = \pm 1$ for any edge $(u, v)$,
\item $\chi(u) < \max\{\chi(v), \chi(w)\}$ for any edges $(u,v)$ and $(u,w)$ with $v \neq w$,
\item $\chi$ is bounded above, 
\item $\chi^{-1}(k)$ is finite for any $k \in \Q$, and
\item $\# \chi^{-1}(k)=1$ for $k \ll 0$.
\end{itemize}
Note that, since $R$ is connected, the difference in the values of $\chi$ at any two vertices is an integer.

To every graded root $R$ we can associate an $\ff[U]$-module $\He^-(R)$, with one generator for each $v \in \Vert(R)$, and relations
$$ U \cdot v = w \text{ if } (v, w) \text{ is an edge, and } \chi(v) - \chi(w)=1.$$
We define the degree of $v \in \Vert(R)$ to be $2\chi(v)$, so that $U$ is a map of degree $-2$. It will usually be convenient for us to think of $R$ in terms of $\He^-(R)$ and view the vertices of $R$ being graded by their degree, rather than the function $\chi$. Thus (for example), when we speak of a graded root shifted by $\sigma$, we mean that the grading $\chi$ is shifted by $\sigma/2$, and when we construct graded roots, we will sometimes use the terms ``degree" and ``grading" interchangeably to mean $2\chi$. This unfortunate overlap of terminology is due to the fact that $2\chi$, rather than $\chi$, is more naturally related to the Maslov grading in Floer homology. 

Let us also mention that by reflecting a graded root $R$ across the horizontal line of grading $-1$ yields a downwards opening graded root $R^r$, of the kind considered in \cite{NemethiOS}. There is also an associated $\ff[U]$-module $\He^+(R)$. Concretely, $\He^+(R)$ has one generator for each $v \in \Vert(R)$, in grading $-2-2\chi(v)$, and has relations
$$ U \cdot w = v \text{ if } (v, w) \text{ is an edge, and } \chi(v) - \chi(w)=1.$$

Now let $G$ be a plumbing graph as in Section~\ref{sec:pg}, and pick $k \in \Char(G)$. Following N{\'e}methi in \cite{NemethiOS, Nem} (cf. also \cite[Section 1.2]{Dai}), we can associate to $(G, k)$ a graded root $R_k$ as follows. We consider the lattice $L=H_2(W; \Z)$, and the function $\chi_k: L \to \Z$ from \eqref{eq:chik}. Let us identify $L$ with $\Z^s$ using the basis coming from the vertices of $G$. Let $\cQ_q$ be the set of $q$-dimensional side-length-one lattice cubes in $\Z^s$. We define a weight function 
$$ w_q: \cQ_q \to \Z$$
by letting  $w_q(\Box_q)$ be the maximum of $\chi_k(x)$ over all vertices $x$ of $\Box_q$. Set
$$ S_{n} = \bigcup_q \{\Box_q \in \cQ_q \mid w_q(\Box_q) \leq n\}.$$
Note that there is an inclusion $S_{n-1} \hookrightarrow S_n$. 

We let the vertices $v$ of $R_k$ in grading $-n$ be the connected components of $S_n$. Further, for $v \in S_{n-1}$ and $w \in S_n$, we let $(v,w)$ be an edge of $R_k$ if and only if $v \subseteq w$.

We define the {\em lattice homology} $\He^-(G, k)$ to be $\He^-(R_k)$. 

Building on the work of Ozsv\'ath and Szab\'o \cite{Plumbed}, N\'emethi proved:
\begin{theorem}[N\'emethi \cite{NemethiOS}]
\label{thm:latticeisom}
Let $G$ be an AR plumbing graph and $Y(G)$ be the corresponding three-manifold. For $k \in \Char(G)$, let 
$$\sigma = \sigma(G, k) = -\frac{|G| + k^2}{4},$$
where $|G|$ denotes the number of vertices of $G$.

Then, the Heegaard Floer homology of $Y(G)$ is given by
\begin{equation}
\label{eq:HFmAR}
\HFm(Y(G), [k]) \cong \He^-(R_k)[\sigma+2].
\end{equation}
In particular, the $d$-invariants of $Y(G)$ are given by
$$ d(Y(G), [k]) = \frac{|G| + k^2}{4} - 2\min \chi_k= \max_{k' \in [k]}\frac{|G| + (k')^2}{4}.$$
\end{theorem}

Note that $\HFm$ of an AR plumbed three-manifold is always supported in degrees congruent to $\sigma$ modulo $2\Z$.

We will sometimes abuse notation slightly and denote the shifted root $R_k[\sigma + 2]$ simply by $R$ when no confusion is possible. We refer to $R = R_k[\sigma + 2]$ as the {\em graded root associated to} $\HFm(Y(G), \s)$. Note that this depends only on $[k] \in \Spinc(Y(G))$, not on $k$ itself, as indeed we have
$$ \HFm(Y(G), \s) \cong \He^-(R).$$
Since the highest degree of a nonzero element in $\HFm$ is $d(Y(G), [k])-2$, the $d$-invariant is the maximum degree of a vertex in $R$, plus two.

With regard to the plus versions, in view of \eqref{eq:dualHF}, we have isomorphisms
\begin{equation}
\label{eq:HFpAR}
\HFp(-Y(G), [k]) \cong \He^+(R) \cong \He^+(R_k)[-\sigma-2].
\end{equation}

\begin{example}
Consider the Brieskorn sphere $Y=\Sigma(2,7,15)$, which is $-1$ surgery on the torus knot $T_{2,7}$. The manifold $Y$ can be represented as the boundary of a plumbing on the AR graph
$$\begin{picture}(0,0)%
\includegraphics{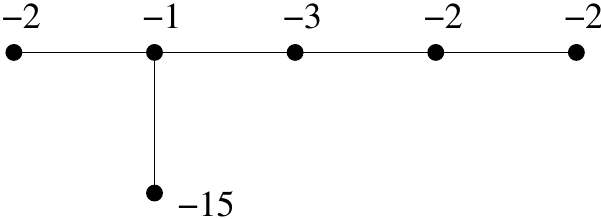}%
\end{picture}%
\setlength{\unitlength}{2960sp}%
\begingroup\makeatletter\ifx\SetFigFont\undefined%
\gdef\SetFigFont#1#2#3#4#5{%
  \reset@font\fontsize{#1}{#2pt}%
  \fontfamily{#3}\fontseries{#4}\fontshape{#5}%
  \selectfont}%
\fi\endgroup%
\begin{picture}(3829,1395)(1111,-2626)
\end{picture}%
$$
The graded root associated to $\HFm(\Sigma(2, 7, 15))$ was computed in \cite{NemethiGRS, Tweedy}, and is shown in Figure~\ref{2715root}. The $d$-invariant is zero.
\end{example}
 
 We now introduce the following concept.
\begin{definition}
A {\em symmetric graded root} is a graded root $R$ together with an involution $J : \Vert(R) \to \Vert(R)$ such that
\begin{itemize}
\item $ \chi(v) = \chi(Jv)$ for any vertex $v$,
\item $(v, w)$ is an edge in $R$ if and only if $(Jv, Jw)$ is an edge on $R$,
\item for every $k \in \Q$, there is at most one $J$-invariant vertex $v$ with $\chi(v) = k$.
\end{itemize} 
\end{definition}

We can represent every symmetric graded root by a planar tree that is symmetric about the vertical axis (the infinite downwards stem). The involution $J$ is the reflection about this axis.

Note that, for a symmetric graded root, the involution $J$ induces an involution (still denoted $J$) on the corresponding lattice homologies $\He^-(R_k)$, $\He^+(R_k)$. The action of $J$ there commutes with that of the variable $U$.
 
Symmetric graded roots appear naturally when we consider self-conjugate $\spinc$ structures on an AR plumbed three-manifold $Y(G)$. Let $\s = [k]$ be such a $\spinc$ structure. As before, let $L$ be the integer lattice spanned by the vertices of $G$. There is an obvious involution $J_0$ on $L$ given by
$$ J_0x = -x - k, $$
where $k$ is viewed as lying in the lattice $L'$. Since $\s$ is self-conjugate, $J_0$ maps $L$ into itself, and it is easily checked that the weight function $w_q$ is invariant under $J_0$. Thus, $J_0$ induces an involution on the set of connected components of each set $S_n$, and hence an involution on the graded root $R_k$. This makes $(R_k, J_0)$ into a symmetric graded root. See \cite[Section 2.1]{Dai} for further discussion.
 

\section{The involution on Heegaard Floer homology}
\label{sec:J}
In Section~\ref{sec:hfi} we introduced the homotopy involution $\inv$ on the Heegaard Floer complexes of three-manifolds equipped with self-conjugate $\spinc$ structures. The goal of this section will be to identify the action induced by $\inv$ at the level of Heegaard Floer homology, for AR plumbed three-manifolds.

\begin{theorem}
\label{thm:Jequiv}
Let $Y = Y(G)$ be an AR plumbed three-manifold and let $\s = [k]$ be a self-conjugate $\spinc$ structure on $Y$. Then, under the isomorphism \eqref{eq:HFmAR} from Theorem~\ref{thm:latticeisom}, the induced action $\inv_*$ of $\inv$ on $\HFm(Y, \s)$ coincides with the reflection involution $J_0$ on $\He^-(R_k)[\sigma+2]$. 
\end{theorem}

\begin{proof}
We begin by reviewing the construction of the isomorphism between Heegaard Floer homology and lattice homology. For convenience, we work with the plus versions, as in the original picture of \cite{Plumbed, NemethiOS}; that is, we will focus on the isomorphisms \eqref{eq:HFpAR}. 
 
The plumbing graph $G$ gives an oriented cobordism $W$ between $-Y$ and $S^3$. 
Let $\Char_\s(G) \subset \Char(G)$ denote the set of characteristic vectors on $W$ limiting to $\s$ on $-Y$. We denote a typical element of $\Char_\s(G)$ by $k$, with $[k] = \s$ being fixed.

Let $\Tower_0$ be the graded $\ff[U]$-module $\HFp(S^3) = \ff[U, U^{-1}]/U\ff[U]$. For any element $x \in \HFp(-Y, \s)$, define a map $\phi_x: \Char_\s(G) \rightarrow \Tower_0$ by
\[
\phi_x(k) = F^+_{W, k} (x),
\]
where $F^+_{W, k}$ is the Heegaard Floer cobordism map associated to the $\spinc$ structure on $W$ with $c_1= k$. In \cite{Plumbed}, it is shown that for any $x$, the map $\phi_x$ satisfies a set of ``adjunction relations" relating the values of $\phi_x$ on different characteristic vectors in $\Char_\s(G)$. Specifically, 
for every $v \in \Vert(G) \subset H_2(W; \Z)$, if we set $n=(k(v) + \langle v, v \rangle)/2,$ we have
\[
U^n \cdot \phi_x(k + 2PD[v]) = \phi_x(k), \ \text{ for } \ n \geq 0
\]
and
\[
 \phi_x(k + 2PD[v]) = U^{-n} \cdot \phi_x(k), \ \text{ for } \ n < 0.
\]

The set of all maps from $\Char_\s(G)$ to $\Tower_0$ satisfying these relations forms a graded $\ff[U]$-module, which can be identified with the lattice homology $\He^+(R_k)[-\sigma-2]$; cf. Proposition 4.7 in \cite{NemethiOS}. Under this identification, the reflection involution $J_0$ on $R_k$ corresponds to 
precomposing a map $\phi_x \in \mathbb{H}^+(G, \s)$ with the reflection $k \mapsto -k$ on $\Char_\s(G)$. 

Theorem 8.3 in \cite{NemethiOS} says that the correspondence $x \mapsto \phi_x$ effects a graded $\ff[U]$-module isomorphism between $\HFp(-Y, \s)$ and the lattice homology $\He^+(R_k)[-\sigma-2]$, as in \eqref{eq:HFpAR}. 

Unwinding the definitions, to check that $\inv_*$ on $\HFp(-Y, \s)$ corresponds to $J_0$ under these isomorphisms
 is equivalent to proving that
 $$\phi_{x}(k) = \phi_{\inv_*(x)}(-k)$$ for all $x \in \HFp(-Y, \s)$ and $k \in \Char_\s(G)$. In other words, we must establish the equality 
\[
F^+_{W, k} = F^+_{W,-k} \circ \inv_*.
\]
This follows from \cite[Theorem 3.6]{HolDiskFour}, which says that the cobordism maps commute with the conjugation involutions. We are also using here the fact that the involution on $\HFp(S^3)$ is the identity, which is clear because $\HFp_d(S^3)$ has dimension at most one in every degree $d$.

We have thus identified the involution $\inv_*$ on $\HFp(-Y, \s)$ with the reflection symmetry on the graded root. Recall that the plus version of Heegaard Floer homology for $-Y$ is related to the minus version of Heegaard Floer cohomology for $Y$, via the duality isomorphism \eqref{eq:dualHF}. Furthermore, it is proved in \cite[Section 4.2]{HMinvolutive} that the involutions $\inv_*$ are taken to each other under this duality isomorphism. The desired characterization of the involution for the minus version follows from here.
\end{proof}

\begin{remark}
Theorem~\ref{thm:Jequiv} is the analogue in the Heegaard Floer setting of the equivariance established by the first author in \cite{Dai} for the $j$-involution on monopole Floer homology.
\end{remark}

\begin{remark}
In the proof of Theorem~\ref{thm:Jequiv} we made use of Theorem 3.6 from \cite{HolDiskFour}, the 
commutation of cobordism maps with the conjugation involutions. The proof of that theorem relies, in particular, on \cite[Lemma 5.2]{HolDiskFour}, which identifies the cobordism maps induced by left-subordinate and right-subordinate triple Heegaard diagrams. In turn, the proof of \cite[Lemma 5.2]{HolDiskFour} uses the cancellation of two-handle maps with three-handle maps, which is an instance of the invariance statement for cobordism maps (that they depend only on the cobordism, not on the handle decompositions and Heegaard diagrams being used). A complete proof of this invariance statement, based on the arguments of Ozsv\'ath and Szab\'o from \cite{HolDiskFour} and the naturality results of Juh\'asz-Thurston \cite{Naturality}, was given by Zemke in \cite{Zemke2}. Note that, for the plus and minus versions, the cobordism maps do not depend on the choice of a path between the basepoints (whereas for the hat version they do); see \cite[Theorem G]{Zemke2}.
\end{remark}


\section{Standard complexes associated to graded roots}
\label{sec:standard}
In this section, we construct a model chain complex for the involutive Floer homology of an AR plumbed three-manifold. We begin by defining a particularly simple model for $\He^-(R)$, where $R$ is an arbitrary graded root as defined in Section~\ref{sec:lattice}. Let $v_1, v_2, \ldots, v_n$ be the leaves of $R$, enumerated in left-to-right lexicographic order. We also enumerate by $\alpha_1, \alpha_2, \ldots, \alpha_{n-1}$ the $n-1$ upward-opening angles in left-to-right lexicographic order. More precisely, let $v_i$ and $v_{i+1}$ be any two consecutive leaves in our enumeration, so that there are unique paths $\gamma_i$ and $\gamma_{i+1}$ in $R$ that start at $v_i$ and $v_{i+1}$, respectively, and run down the entire length of the infinite stem. These two paths merge at a unique vertex of highest grading, and this vertex supports a unique angle whose sides lie on $\gamma_i$ and $\gamma_{i+1}$, which we denote by $\alpha_i$. The case of $\HFm(\Sigma(2, 7, 15))$ is shown below in Figure \ref{2715root}. 

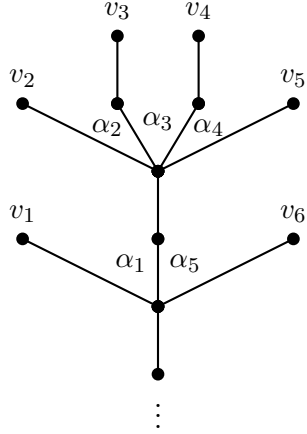
\begin{figure}[h!]
\begin{tikzpicture}[thick,scale=0.9]%
	\node[text width=0.1cm] at (0, -3.5) {\vdots};
	\draw (0, 0) node[circle, draw, fill=black!100, inner sep=0pt, minimum width=4pt] {} -- (0, -1) node[circle, draw, fill=black!100, inner sep=0pt, minimum width=4pt] {};
	\draw (0, -1) node[circle, draw, fill=black!100, inner sep=0pt, minimum width=4pt] {} -- (0, -2) node[circle, draw, fill=black!100, inner sep=0pt, minimum width=4pt] {};
	\draw (0, -2) node[circle, draw, fill=black!100, inner sep=0pt, minimum width=4pt] {} -- (0, -3) node[circle, draw, fill=black!100, inner sep=0pt, minimum width=4pt] {};
	
	\draw (0, 0) node[circle, draw, fill=black!100, inner sep=0pt, minimum width=4pt] {} -- (-0.6, 1) node[circle, draw, fill=black!100, inner sep=0pt, minimum width=4pt] {};
	\draw (0, 0) node[circle, draw, fill=black!100, inner sep=0pt, minimum width=4pt] {} -- (0.6, 1) node[circle, draw, fill=black!100, inner sep=0pt, minimum width=4pt] {};
	
	\draw (-0.6, 1) node[circle, draw, fill=black!100, inner sep=0pt, minimum width=4pt] {} -- (-0.6, 2) node[circle, draw, fill=black!100, inner sep=0pt, minimum width=4pt, label = $v_3$] {};
	\draw (0.6, 1) node[circle, draw, fill=black!100, inner sep=0pt, minimum width=4pt] {} -- (0.6, 2) node[circle, draw, fill=black!100, inner sep=0pt, minimum width=4pt, label = $v_4$] {};	
	
	\draw (0, 0) node[circle, draw, fill=black!100, inner sep=0pt, minimum width=4pt] {} -- (-2, 1) node[circle, draw, fill=black!100, inner sep=0pt, minimum width=4pt, label = $v_2$] {};
	\draw (0, 0) node[circle, draw, fill=black!100, inner sep=0pt, minimum width=4pt] {} -- (2, 1) node[circle, draw, fill=black!100, inner sep=0pt, minimum width=4pt, label = $v_5$] {};
	
	\draw (0, -2) node[circle, draw, fill=black!100, inner sep=0pt, minimum width=4pt] {} -- (-2, -1) node[circle, draw, fill=black!100, inner sep=0pt, minimum width=4pt, label = $v_1$] {};
	\draw (0, -2) node[circle, draw, fill=black!100, inner sep=0pt, minimum width=4pt] {} -- (2, -1) node[circle, draw, fill=black!100, inner sep=0pt, minimum width=4pt, label = $v_6$] {};
	\node[label = $\alpha_3$] at (0,0.3) {};
	\node[label = $\alpha_2$] at (-0.75,0.2) {};
	\node[label = $\alpha_4$] at (0.75,0.2) {};
	\node[label = $\alpha_1$] at (-0.4,-1.8) {};
	\node[label = $\alpha_5$] at (0.4,-1.8) {};
\end{tikzpicture}
\caption{Graded root (with leaf and angle labels) for $\HFm(\Sigma(2, 7, 15))$. \hspace{\textwidth} Uppermost vertices have Maslov grading $-2$; successive rows have grading difference two.}\label{2715root}
\end{figure}

We construct a free and finitely-generated model chain complex with the same homology and $\ff[U]$-structure as $\He^-(R)$. Let $\gr(v_i) = 2\chi(v_i)$ denote the degree of vertex $v_i$, and let $\gr(\alpha_i)$ denote the degree of the vertex supporting angle $\alpha_i$. The generators (over $\ff[U]$) of our complex are given as follows. For each vertex $v_i$, we place a single generator in grading $\gr(v_i)$, which by abuse of notation we also denote by $v_i$. Note that $v_i$ is a generator of our complex over $\ff[U]$, so that as an abelian group we are introducing an entire tower of generators $\ff[U]v_i$. For each angle $\alpha_i$, we similarly place a single generator in grading $\gr(\alpha_i) + 1$, denoting this by $\alpha_i$. This gives a splitting of our complex into chains spanned by $U$-powers of the $v_i$, which we refer to as even chains, and chains spanned by $U$-powers of the $\alpha_i$, which we refer to as odd chains. 

We now define our differential to be identically zero on the even chains, and set
\[
\partial \alpha_i = U^{(\gr(v_i)-\gr(\alpha_i))/2}v_i + U^{(\gr(v_{i+1})-\gr(\alpha_i))/2}v_{i+1}
\]
on the $\alpha_i$, extending to the entire complex linearly and $U$-equivariantly. Note that since $\partial$ takes each $\alpha_i$ to the sum of ($U$-powers of) $v_i$ and $v_{i+1}$, it follows from the fact that the $\alpha_i$ are ordered from $\alpha_1$ to $\alpha_{n-1}$ that $\partial$ is injective on the odd chains in $C_*(R)$. We call this complex the \textit{standard complex associated to $R$} and denote it by $C_*(R)$. The standard complex associated to the graded root given in Figure \ref{2715root} is shown in Figure \ref{figex}. 

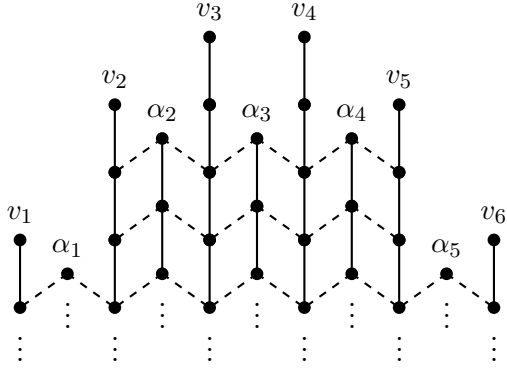
\begin{figure}[h!]
\begin{tikzpicture}[thick,scale=0.9]%

	
	\draw (-0.7, 0) node[circle, draw, fill=black!100, inner sep=0pt, minimum width=4pt, label = $v_3$] {} -- (-0.7, -1) node[circle, draw, fill=black!100, inner sep=0pt, minimum width=4pt] {};
	\draw (-0.7, -1) node[circle, draw, fill=black!100, inner sep=0pt, minimum width=4pt] {} -- (-0.7, -2) node[circle, draw, fill=black!100, inner sep=0pt, minimum width=4pt] {};
	\draw (-0.7, -2) node[circle, draw, fill=black!100, inner sep=0pt, minimum width=4pt] {} -- (-0.7, -3) node[circle, draw, fill=black!100, inner sep=0pt, minimum width=4pt] {};
	\draw (-0.7, -3) node[circle, draw, fill=black!100, inner sep=0pt, minimum width=4pt] {} -- (-0.7, -4) node[circle, draw, fill=black!100, inner sep=0pt, minimum width=4pt] {};
	\node[text width=0.1cm] at (-0.7, -4.5) {\vdots};
	
	\draw (0.7, 0) node[circle, draw, fill=black!100, inner sep=0pt, minimum width=4pt, label = $v_4$] {} -- (0.7, -1) node[circle, draw, fill=black!100, inner sep=0pt, minimum width=4pt] {};
	\draw (0.7, -1) node[circle, draw, fill=black!100, inner sep=0pt, minimum width=4pt] {} -- (0.7, -2) node[circle, draw, fill=black!100, inner sep=0pt, minimum width=4pt] {};
	\draw (0.7, -2) node[circle, draw, fill=black!100, inner sep=0pt, minimum width=4pt] {} -- (0.7, -3) node[circle, draw, fill=black!100, inner sep=0pt, minimum width=4pt] {};
	\draw (0.7, -3) node[circle, draw, fill=black!100, inner sep=0pt, minimum width=4pt] {} -- (0.7, -4) node[circle, draw, fill=black!100, inner sep=0pt, minimum width=4pt] {};
	\node[text width=0.1cm] at (0.7, -4.5) {\vdots};
	
	\draw (-2.1, -1) node[circle, draw, fill=black!100, inner sep=0pt, minimum width=4pt, label = $v_2$] {} -- (-2.1, -2) node[circle, draw, fill=black!100, inner sep=0pt, minimum width=4pt] {};
	\draw (-2.1, -2) node[circle, draw, fill=black!100, inner sep=0pt, minimum width=4pt] {} -- (-2.1, -3) node[circle, draw, fill=black!100, inner sep=0pt, minimum width=4pt] {};
	\draw (-2.1, -3) node[circle, draw, fill=black!100, inner sep=0pt, minimum width=4pt] {} -- (-2.1, -4) node[circle, draw, fill=black!100, inner sep=0pt, minimum width=4pt] {};
	\node[text width=0.1cm] at (-2.1, -4.5) {\vdots};
	
	\draw (2.1, -1) node[circle, draw, fill=black!100, inner sep=0pt, minimum width=4pt, label = $v_5$] {} -- (2.1, -2) node[circle, draw, fill=black!100, inner sep=0pt, minimum width=4pt] {};
	\draw (2.1, -2) node[circle, draw, fill=black!100, inner sep=0pt, minimum width=4pt] {} -- (2.1, -3) node[circle, draw, fill=black!100, inner sep=0pt, minimum width=4pt] {};
	\draw (2.1, -3) node[circle, draw, fill=black!100, inner sep=0pt, minimum width=4pt] {} -- (2.1, -4) node[circle, draw, fill=black!100, inner sep=0pt, minimum width=4pt] {};
	\node[text width=0.1cm] at (2.1, -4.5) {\vdots};
	
	\draw (-3.5, -3) node[circle, draw, fill=black!100, inner sep=0pt, minimum width=4pt, label = $v_1$] {} -- (-3.5, -4) node[circle, draw, fill=black!100, inner sep=0pt, minimum width=4pt] {};
	\node[text width=0.1cm] at (-3.5, -4.5) {\vdots};
	
	\draw (3.5, -3) node[circle, draw, fill=black!100, inner sep=0pt, minimum width=4pt, label = $v_6$] {} -- (3.5, -4) node[circle, draw, fill=black!100, inner sep=0pt, minimum width=4pt] {};
	\node[text width=0.1cm] at (3.5, -4.5) {\vdots};
	

	\draw (0, -1.5) node[circle, draw, fill=black!100, inner sep=0pt, minimum width=4pt, label = $\alpha_3$] {} -- (0, -2.5) node[circle, draw, fill=black!100, inner sep=0pt, minimum width=4pt] {};
	\draw (0, -2.5) node[circle, draw, fill=black!100, inner sep=0pt, minimum width=4pt] {} -- (0, -3.5) node[circle, draw, fill=black!100, inner sep=0pt, minimum width=4pt] {};
	\node[text width=0.1cm] at (0, -4) {\vdots};	
	
	\draw (0, -1.5)  -- (-0.7, -2) [dashed];
	\draw (0, -1.5)  -- (0.7, -2) [dashed];
	\draw (0, -2.5)  -- (-0.7, -3) [dashed];
	\draw (0, -2.5)  -- (0.7, -3) [dashed];
	\draw (0, -3.5)  -- (-0.7, -4) [dashed];
	\draw (0, -3.5)  -- (0.7, -4) [dashed];

	\draw (-1.4, -1.5) node[circle, draw, fill=black!100, inner sep=0pt, minimum width=4pt, label = $\alpha_2$] {} -- (-1.4, -2.5) node[circle, draw, fill=black!100, inner sep=0pt, minimum width=4pt] {};
	\draw (-1.4, -2.5) node[circle, draw, fill=black!100, inner sep=0pt, minimum width=4pt] {} -- (-1.4, -3.5) node[circle, draw, fill=black!100, inner sep=0pt, minimum width=4pt] {};
	\node[text width=0.1cm] at (-1.4, -4) {\vdots};	
	
	\draw (-1.4, -1.5)  -- (-2.1, -2) [dashed];
	\draw (-1.4, -1.5)  -- (-0.7, -2) [dashed];
	\draw (-1.4, -2.5)  -- (-2.1, -3) [dashed];
	\draw (-1.4, -2.5)  -- (-0.7, -3) [dashed];
	\draw (-1.4, -3.5)  -- (-2.1, -4) [dashed];
	\draw (-1.4, -3.5)  -- (-0.7, -4) [dashed];

		\draw (1.4, -1.5) node[circle, draw, fill=black!100, inner sep=0pt, minimum width=4pt, label = $\alpha_4$] {} -- (1.4, -2.5) node[circle, draw, fill=black!100, inner sep=0pt, minimum width=4pt] {};
	\draw (1.4, -2.5) node[circle, draw, fill=black!100, inner sep=0pt, minimum width=4pt] {} -- (1.4, -3.5) node[circle, draw, fill=black!100, inner sep=0pt, minimum width=4pt] {};
	\node[text width=0.1cm] at (1.4, -4) {\vdots};	
	
	\draw (1.4, -1.5)  -- (2.1, -2) [dashed];
	\draw (1.4, -1.5)  -- (0.7, -2) [dashed];
	\draw (1.4, -2.5)  -- (2.1, -3) [dashed];
	\draw (1.4, -2.5)  -- (0.7, -3) [dashed];
	\draw (1.4, -3.5)  -- (2.1, -4) [dashed];
	\draw (1.4, -3.5)  -- (0.7, -4) [dashed];

	\draw (-2.8, -3.5) node[circle, draw, fill=black!100, inner sep=0pt, minimum width=4pt, label = $\alpha_1$] {};
	\node[text width=0.1cm] at (-2.8, -4) {\vdots};
	\draw (-2.8, -3.5)  -- (-2.1, -4) [dashed];
	\draw (-2.8, -3.5)  -- (-3.5, -4) [dashed];

	\draw (2.8, -3.5) node[circle, draw, fill=black!100, inner sep=0pt, minimum width=4pt, label = $\alpha_5$] {};
	\node[text width=0.1cm] at (2.8, -4) {\vdots};
	\draw (2.8, -3.5)  -- (2.1, -4) [dashed];
	\draw (2.8, -3.5)  -- (3.5, -4) [dashed];

\end{tikzpicture}
\caption{Standard complex associated to the $\ff[U]$-module $\HFm(\Sigma(2, 7, 15))$. \hspace{\textwidth}Solid lines represent the action of $U$; dashed lines represent the action of $\partial$.}
\label{figex}
\end{figure}

It is easily verified that $C_*(R)$ replicates the homology and $\ff[U]$-structure given by $\He^-(R)$. Roughly speaking, the idea is that $R$ may be formed by taking $n$ infinite disjoint strands corresponding to the paths $\gamma_1, \gamma_2, \ldots, \gamma_n$ and gluing together each pair $\gamma_i$ and $\gamma_{i+1}$ below some grading determined by the branch points of $R$. On the side of $C_*(R)$, this corresponds to the fact that generators in the tower $\ff[U]v_i$ are in bijection with the vertices of $\gamma_i$; the $\alpha_i$ serve to ``glue together" (on the level of homology) generators in different towers. The reader may find it useful to compare Figures \ref{2715root} and \ref{figex}. A precise version of this argument is given below.
\begin{lemma}\label{lemgraded1}
Let $R$ be a graded root. Then the homology $H_*(R)$ of the chain complex $C_*(R)$ is isomorphic to $\He^-(R)$ as a $\ff[U]$-module.
\end{lemma}
\begin{proof}
First observe that since $\partial$ is injective on the odd part of $C_*(R)$, the homology $H_*(R)$ is supported only in even degrees. Thus, consider an even chain of the form $U^nv_i$. This is in fact a cycle, since $\partial$ is zero on even chains. Let $\gamma_i$ be the unique path in $R$ that starts at $v_i$ and runs down the infinite stem of $R$. Our proposed isomorphism takes the homology class of $U^nv_i$ in $C_*(R)$ to the generator of $\He^-(R)$ corresponding to the unique vertex in $R$ that lies on $\gamma_i$ and is distance $n$ from $v_i$. (That is, there is a sub-path of $\gamma_i$ consisting of $n$ edges which joins $v_i$ and the desired vertex.) Said differently, we map the homology class of $v_i$ in $C_*(R)$ to the element of $\He^-(R)$ corresponding to the vertex $v_i$ in $R$, and extend $U$-equivariantly to generators in the tower $\ff[U]v_i$. 

In order to check that this is well-defined, we must check that if $U^av_i$ and $U^bv_j$ are homologous in $C_*(R)$, then the above procedure yields the same vertex in $R$ for both. It suffices to establish the case when $j = i + 1$ and the two cycles $U^av_i$ and $U^bv_{i+1}$ are homologous via a single boundary element $\partial U^c\alpha_i$, since in general we simply sum together such cases. By definition, the existence of $\alpha_i$ means that the paths $\gamma_i$ and $\gamma_{i+1}$ coincide in all gradings less than or equal to $\gr(\alpha_i)$. Since the gradings of the two cycles in question are given by $\gr(\alpha_i) - 2c$ with $c \geq 0$, this proves the claim. 

Our correspondence is clearly equivariant with respect to the action of $U$, and we may extend it linearly to all of $C_*(R)$ to obtain a $\ff[U]$-module map from $H_*(R)$ to $\He^-(R)$. It is obviously surjective, since any vertex in $R$ lies on a path from some leaf to the infinite stem. Moreover, as a map defined on the cycles of $C_*(R)$, our correspondence has kernel spanned precisely by sums of the form $\partial U^c\alpha_i = U^av_i + U^bv_{i+1}$, as above. This completes the proof.
\end{proof}

Now suppose that $R$ is a symmetric graded root. There is an obvious involution on $C_*(R)$ given by sending $v_i$ to $v_{n-i+1}$, $\alpha_i$ to $\alpha_{n-i}$, and extending linearly and $U$-equivariantly. Since this evidently induces the reflection map on $H_*(R) \cong \He^-(R)$ described in Section~\ref{sec:lattice}, by a slight abuse of notation we similarly denote this involution and its induced action on homology by $J_0$. We claim that $J_0$ is standard in the sense that any other chain map on $C_*(R)$ which induces the reflection involution on $H_*(R)$ must be chain homotopic to it. In general, of course, two maps that induce the same map on homology need not be chain homotopic, but in our case the extreme simplicity of $C_*(R)$ allows us to tautologically construct the chain homotopy:
\begin{lemma}\label{lemgraded2}
Let $J_0$ be the involution on $C_*(R)$ given by sending $v_i$ to $v_{n-i+1}$, $\alpha_i$ to $\alpha_{n-i}$, and extending linearly and $U$-equivariantly. Any other chain map $J$ on $C_*(R)$ which induces the reflection involution on $H_*(R) \cong \He^-(R)$ is chain homotopic to $J_0$.
\end{lemma}
\begin{proof}
We construct a chain homotopy $H$ between $J$ and $J_0$ as follows. Let $v$ be an even chain in $C_*(R)$. Since $v$ is automatically a cycle and $J$ and $J_0$ induce the same map on homology, we know that $Jv$ and $J_0v$ differ by a boundary. Because $\partial$ has zero kernel on the odd part of $C_*(R)$, this boundary has a unique preimage under $\partial$, which we denote by $Hv$. Thus we have $Jv + J_0v = \partial Hv$. Note that since $J$ and $J_0$ are linear and $U$-equivariant, the uniqueness of $Hv$ shows that $H$ is a linear and $U$-equivariant map from the even chains to the odd chains of $C_*(R)$. We define our chain homotopy to be equal to $H$ on the even part of $C_*(R)$ and identically zero on the odd part of $C_*(R)$.  

Because $\partial$ is identically zero on the even part of $C_*(R)$, the above equality establishes the desired chain homotopy identity on even chains. Thus, let $\alpha$ be an odd chain. Putting $v = \partial \alpha$ yields the equality $J\partial \alpha + J_0\partial \alpha = \partial H\partial \alpha$. Since $J$ and $J_0$ are chain maps, the left-hand side is equal to $\partial(J\alpha + J_0\alpha)$, and since $\partial$ has zero kernel on the odd part of $C_*(R)$, this shows that $J\alpha + J_0\alpha = H\partial \alpha$. Having defined $H$ to be zero on odd chains, this establishes the chain homotopy identity for odd chains, completing the proof.
\end{proof}

\begin{remark}
In the above proof, we did not actually use the fact that $J$ or $J_0$ induced the reflection involution. More generally, let $f$ and $g$ be two chain maps between standard complexes of graded roots. Then if $f$ and $g$ coincide on the level of homology, the above construction shows that $f$ and $g$ are chain homotopic. 
\end{remark}

Now let $Y$ be an AR plumbed three-manifold and $\s = [k]$ be a self-conjugate $\spinc$ structure on $Y$. For convenience, denote by $R$ the (shifted) graded root $R_k[\sigma+2]$, so that $\HFm(Y, \s) \cong \He^-(R)$. We claim that the pairs $(\CFm(Y, \s), \inv)$ and $(C_*(R), J_0)$ are equivalent $\inv$-complexes. We begin by first showing that the complexes $\CFm(Y, \s)$ and $C_*(R)$ are homotopy equivalent over $\ff[U]$. This is a consequence of the following standard algebraic lemma, whose proof is just a re-phrasing of the fact that any two free resolutions of the same module are homotopy equivalent.

\begin{lemma}\label{lemequiv1}
Let $C_1$ and $C_2$ be two chain complexes which are free over a principal ideal domain. If $C_1$ and $C_2$ have the same homology, then they are (non-canonically) homotopy equivalent.
\end{lemma}
\begin{proof}
Let $H_1$ and $H_2$ be the homologies of $(C_1, \partial_1)$ and $(C_2, \partial_2)$, respectively. Because submodules of free modules over a principal ideal domain are free, we may choosing splittings $C_1 = Z_1 \oplus B_1$ and $C_2 = Z_2 \oplus B_2$, where $Z_i = \ker \partial_i \subseteq C_i$ and the $B_i$ map isomorphically onto $\im \partial_i \subseteq C_i$ via $\partial_i$. These fit into free resolutions
\[
0 \rightarrow B_i \xrightarrow{\partial_i} Z_i \rightarrow H_i \rightarrow 0
\]
for $i = 1, 2$. If we fix an isomorphism $\phi$ between $H_0$ and $H_1$, then a standard diagram chase allows us to extend $\phi$ to a commutative diagram
\[
\begin{tikzpicture}[description/.style={fill=white,inner sep=2pt}]
\matrix (m) [matrix of math nodes, row sep=2.5em,
column sep=3em, text height=1.5ex, text depth=0.25ex]
{ 
0 & B_1 & Z_1 & H_1 & 0\\
0 & B_2 & Z_2 & H_2 & 0\\
};
\path[->,font=\scriptsize]
(m-1-1) edge node[auto] {} (m-1-2)
(m-1-2) edge node[auto] {$\partial_1$} (m-1-3)
(m-1-3) edge node[auto] {} (m-1-4)
(m-1-4) edge node[auto] {} (m-1-5)

(m-2-1) edge node[auto] {} (m-2-2)
(m-2-2) edge node[auto] {$\partial_2$} (m-2-3)
(m-2-3) edge node[auto] {} (m-2-4)
(m-2-4) edge node[auto] {} (m-2-5)

(m-1-2) edge node[auto] {$g$} (m-2-2)
(m-1-3) edge node[auto] {$f$} (m-2-3)
(m-1-4) edge node[auto] {$\phi$} (m-2-4);
\end{tikzpicture}
\]
This defines a map $f \oplus g: C_1 \rightarrow C_2$. After constructing a similar diagram for $\phi^{-1}$, the usual proof that any two free resolutions of the same module are homotopy equivalent shows that $f \oplus g$ is a homotopy equivalence. Moreover, an examination of the maps constructed shows that if $H_1$ and $H_2$ are graded and $\phi$ preserves the grading, then the maps $f$ and $g$ can also be chosen to preserve grading, and all chain homotopy maps increase grading by one. This completes the proof.
\end{proof}

Since $\CFm(Y, \s)$ and $C_*(R)$ are free over $\ff[U]$, Lemmas \ref{lemgraded1} and \ref{lemequiv1} imply that they are homotopy equivalent. Moreover, we may choose this homotopy equivalence so that the induced isomorphism on homology coincides with the lattice homology isomorphism between $H_*(R) \cong \He^-(R)$ and $\HFm(Y, \s)$.  However, if we want to take into account the structure of $\CFm(Y, \s)$ and $C_*(R)$ as $\inv$-complexes, we should view them as complexes over the free polynomial ring in two variables $\ff[U, \inv]$. Because our complexes are not free over $\ff[U, \inv]$ (nor indeed is $\ff[U, \inv]$ a principal ideal domain), we cannot apply Lemma~\ref{lemequiv1} to identify $J_0$ with $\inv$ on the chain level. Instead, let the homotopy equivalence between $C_*(R)$ and $\CFm(Y, \s)$ be given by the pair of maps $F: C_*(R) \rightarrow \CFm(Y, \s)$ and $G: \CFm(Y, \s) \rightarrow C_*(R)$. Then we can pull back $\inv$ to a map
$$ J = G \circ \inv \circ F: C_*(R) \rightarrow C_*(R). $$
By construction, the pairs $(C_*(R), J)$ and $(\CFm(Y, \s), \inv)$ are equivalent in the sense of Definition~\ref{def:E}. Moreover, according to Theorem~\ref{thm:Jequiv} the homological action of $J$ on $H_*(R)$ is given by the reflection involution, since $F$ and $G$ on the homological level coincide with the lattice homology isomorphism $\He^-(R) \cong \HFm(Y, \s)$. Since this is also the action of $J_0$, Lemma~\ref{lemgraded2} implies that $J$ and $J_0$ are chain homotopic. Hence we have:

\begin{theorem}\label{thm:standard}
Let $Y=Y(G)$ be the plumbed three-manifold associated to an AR graph $G$, oriented as the boundary of the plumbing. Let $\s=[k]$ be a self-conjugate $\spinc$ structure on $Y$. Let also $R_k$ be the graded root associated to $(G, k)$, and let $J_0$ be the reflection involution on $C_*(R_k)[\sigma + 2]$. Then, the pairs $(\CFm(Y, \s), \inv)$ and $(C_*(R_k)[\sigma+2], J_0)$ are equivalent $\inv$-complexes.
\end{theorem}
\begin{proof}
As observed above, the pairs $(C_*(R), J)$ and $(\CFm(Y, \s), \inv)$ are equivalent $\inv$-complexes by Lemmas~\ref{lemgraded1} and \ref{lemequiv1}, together with the construction of $J$. The maps $J$ and $J_0$ are moreover chain homotopic as maps from $C_*(R)$ to itself by Theorem~\ref{thm:Jequiv} and Lemma~\ref{lemgraded2}. This shows that the pairs $(C_*(R), J)$ and $(C_*(R), J_0)$ are equivalent $\inv$-complexes (taking the maps $F$ and $G$ in Definition~\ref{def:E} to be the identity).
\end{proof}

We finally obtain the desired result:
\begin{corollary}\label{lemequiv2}
The involutive complexes associated to the pairs $(\CFm(Y, \s), \inv)$ and $(C_*(R_k)[\sigma+2], J_0)$ are quasi-isomorphic. 
\end{corollary}
\begin{proof}
This follows immediately from Lemma~\ref{lem:EquivCones} and Theorem~\ref{thm:standard}. 
\end{proof}


\section{The involutive Heegaard Floer homology of AR plumbed three-manifolds}
\label{sec:hfiAR}
We now prove Theorem~\ref{thm:main}. Note that since $J_0$ is equivariant with respect to the action of $U$, the expressions\footnote{Recall that we use the same notation $J_0$ to denote the involution on the complex $C_*(R)$ and on its homology. Whenever we write   $\ker (1 + J_0)$, $\im (1 + J_0)$ or $\coker (1 + J_0)$, we refer to the action on homology.} $\ker (1 + J_0)$ and $\coker (1 + J_0)$ inherit graded $\ff[U]$-module structures from $\He^-(R_k)[\sigma+2]$. 

\begin{proof}[Proof of Theorem~\ref{thm:main}]
Denote by $R$ the (shifted) graded root $R_k[\sigma+2]$. By Corollary~\ref{lemequiv2}, the involutive homology of $(C_*(R), J_0)$ is isomorphic to $\HFIm(Y, \s)$. Write the mapping cone complex of $C_*(R)$ as the complex $C_*(R)[-1] \otimes \ff[Q]/Q^2$, together with the differential $\partial'= \partial + Q(1 + J_0)$. Recall that $C_*(R)$ splits into even and odd chains, which we denote by $C_\textit{even}$ and $C_{odd}$, respectively. Expanding the mapping cone complex with respect to this splitting, we obtain the complex
\[
\begin{tikzpicture}[description/.style={fill=white,inner sep=2pt}]
\matrix (m) [matrix of math nodes, row sep=2.5em,
column sep=3.5em, text height=1.5ex, text depth=0.25ex]
{ 
C_\textit{odd}[-1] & QC_\textit{odd}[-1]\\
C_\textit{even}[-1] & QC_\textit{even}[-1]\\
};
\path[->,font=\scriptsize]
(m-1-1) edge node[auto] {$Q(1+J_0)$} (m-1-2)
(m-2-1) edge node[auto] {$Q(1+J_0)$} (m-2-2)

(m-1-1) edge node[auto] {$\partial$} (m-2-1)
(m-1-2) edge node[auto] {$\partial$} (m-2-2);
\end{tikzpicture}
\]
with the action of $\partial'$ on each term given by the sum of all outgoing arrows. Since $\partial$ is injective on $C_\textit{odd}$, we see from this that $\ker\; \partial'$ is the direct sum of $QC_\textit{even}[-1]$ and the submodule $K$ of $C_\textit{even}[-1] \oplus QC_\textit{odd}[-1]$ defined by
\[
K = \{v + Q\alpha : v\in C_\textit{even} \text{ and } \alpha \in C_\textit{odd} \text{ and } (1 + J_0)v = \partial \alpha\}.
\]
After quotienting out by $\im \partial'$, the term $QC_\textit{even}[-1]$ gives us a summand isomorphic to $\coker (1 + J_0)$. To understand what happens to $K$, observe that for a fixed even chain $v$, there exists an odd chain $\alpha$ such that $v + Q\alpha \in K$ if and only if $(1 + J_0)v$ is nullhomologous in $C_*(R)$. If such an $\alpha$ exists, it is moreover unique since $\partial$ is injective on odd chains. Identifying elements $v + Q\alpha \in K$ with their even parts $v$, it follows that after quotienting out by $\im \partial'$, the term $K$ gives rise to a summand isomorphic to $\ker (1 + J_0)$. Moreover, it is clear that under this identification, the action of $Q$ takes $\ker (1 + J_0)$ to $\ker (1 + J_0)/\im(1 + J_0) \subseteq \coker (1 + J_0)$, since a representative of the form $v + Q\alpha$ is mapped simply to $Qv$. This completes the proof.
\end{proof}

Theorem~\ref{thm:main} should be compared with Theorem 2.3 of \cite{Dai}, in which a similar expression is derived for the Pin(2)-equivariant monopole Floer homology. Indeed, in the spirit of the proof given there, a simpler proof of Theorem~\ref{thm:main} can be given using the standard mapping cone sequence relating $\HFIm(Y, \s)$ and $\HFm(Y, \s)$ established in \cite[Proposition 4.6]{HMinvolutive}. As we shall see in a moment, however, the advantage of constructing an actual model chain complex is that we may invoke the connected sum formula to compute the involutive Floer homology of connected sums of AR manifolds. 

\begin{remark}
It is conjectured in \cite[Conjecture 3.2]{HMinvolutive} that, for rational homology spheres $Y$, the involutive Heegaard Floer homology $\HFI^+_*(Y, \s)$ is isomorphic to the $\Z/4$-equivariant Borel homology of the Seiberg-Witten Floer spectrum $\mathit{SWF}(Y, \s)$. Theorem~\ref{thm:main}, combined with the calculations of Stoffregen from \cite{Stoffregen}, confirms the conjecture in the case of rational homology spheres that are Seifert fibered over an orbifold with underlying space $S^2$. Indeed, for such spaces, Stoffregen gives a model for the cellular chain complex of the Seiberg-Floer spectrum, with $\ff$ coefficients, as a module over the cellular complex of the symmetry group $\pin$. From this information, using \cite[Equation (8)]{Stoffregen} for the subgroup $\Z/4 \subset \pin$, one can compute the $\Z/4$-equivariant homology of $\mathit{SWF}(Y, \s)$, and check that it agrees with our calculations. 

Alternatively, one can use \cite[Proposition 3.1]{HMinvolutive}, which says that for a space $X$ with a $\pin$ action, we have an isomorphism
$$ H_*^{\Z/4}(X; \ff) \cong H_*(\Cone(C_*^{S^1}(X; \ff) \xrightarrow{Q (1+\inv) } Q \ccdot C_*^{S^1}(X; \ff)[-1])).$$
This leads to an exact triangle relating $H_*^{S^1}(X; \ff)$ to $H_*^{\Z/4}(X; \ff)$. We have $H_*^{S^1}(\swf(Y, \s); \ff) \cong \HF^+(Y, \s)$ by \cite{LidmanManolescu}. By passing to the coBorel and minus versions, and using the fact that $\HF^-(Y, \s)$ is supported in even degrees, we find that $H^*_{\Z/4}(\swf(Y, \s); \ff)$ is isomorphic to the cone of $1+J_0$ on homology, in agreement with the calculation in involutive Heegaard Floer homology. 
\end{remark}

We now make a brief digression to present an attractive graphical representation of the involutive homology associated to a symmetric graded root $R$. We begin by drawing a graded root for the $\ff[U]$-module $\coker (1 + J_0)$. Passing from $\He^-(R)$ to $\coker (1 + J_0)$ identifies pairs of vertices that are symmetric across the infinite stem, with the caveat that vertices lying on the stem itself remain in singleton equivalence classes. Since this identification is $U$-equivariant, we obtain a graded root whose lattice homology is $\coker (1 + J_0)$ by simply ``folding $R$ in half" across this stem. See Figure \ref{fighfi} for an example. 

Now consider $\ker (1 + J_0)$. As a vector space, this is isomorphic to $\coker (1 + J_0)$, as indeed it has the same rank in each grading. We thus begin by again folding $R$ in half. However, in this case the interpretation of the vertices in the half-root are slightly different. As before, vertices lying on the infinite stem in the half-root are identified with their counterparts in $R$, but vertices not on the infinite stem correspond to sums of symmetric pairs of vertices in $R$. To see the effect that this has on the $U$-structure, suppose that $v$ is a vertex of the latter kind such that in our original graded root, the product $w = Uv$ lies on the infinite stem. Then $J_0w = w$, so $U(v + J_0v) = 2w = 0$. Hence in our half-root representation of $\ker (1 + J_0)$, the vertices $v$ and $w$ are \textit{not} connected by an edge. A graphical representation of $\ker (1 + J_0)$ is thus obtained by folding $R$ in half and deleting each edge coming off of the infinite stem. Again, see Figure \ref{fighfi} for an example. 

Finally, to compute the action of $Q$ from $\ker (1 + J_0)$ to $\coker (1 + J_0)$, observe that an element of $\ker (1 + J_0)$ consisting of the sum of two symmetric vertices maps to zero in $\coker (1 + J_0)$, while a vertex lying on the infinite stem has nonzero image. Thus the action of $Q$ is an isomorphism from the infinite stem of $\ker (1 + J_0)$ onto the infinite stem of $\coker (1 + J_0)$, and is zero everywhere else. The complete construction for $\HFIm(\Sigma(2, 7, 15))$ is given below in Figure~\ref{fighfi}. We refer to such a diagram as a graded root for the involutive homology, even though it is not quite a graded root in the sense defined previously.

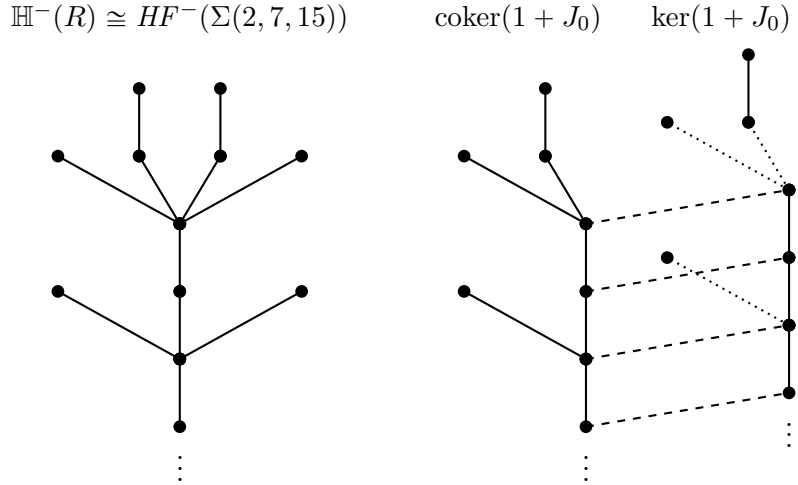
\begin{figure}[h!]
\begin{tikzpicture}[thick,scale=0.9]%
	\node[text width=0.1cm] at (0, -3.5) {\vdots};
	\draw (0, 0) node[circle, draw, fill=black!100, inner sep=0pt, minimum width=4pt] {} -- (0, -1) node[circle, draw, fill=black!100, inner sep=0pt, minimum width=4pt] {};
	\draw (0, -1) node[circle, draw, fill=black!100, inner sep=0pt, minimum width=4pt] {} -- (0, -2) node[circle, draw, fill=black!100, inner sep=0pt, minimum width=4pt] {};
	\draw (0, -2) node[circle, draw, fill=black!100, inner sep=0pt, minimum width=4pt] {} -- (0, -3) node[circle, draw, fill=black!100, inner sep=0pt, minimum width=4pt] {};
	
	\draw (0, 0) node[circle, draw, fill=black!100, inner sep=0pt, minimum width=4pt] {} -- (-0.6, 1) node[circle, draw, fill=black!100, inner sep=0pt, minimum width=4pt] {};
	\draw (0, 0) node[circle, draw, fill=black!100, inner sep=0pt, minimum width=4pt] {} -- (0.6, 1) node[circle, draw, fill=black!100, inner sep=0pt, minimum width=4pt] {};
	
	\draw (-0.6, 1) node[circle, draw, fill=black!100, inner sep=0pt, minimum width=4pt] {} -- (-0.6, 2) node[circle, draw, fill=black!100, inner sep=0pt, minimum width=4pt] {};
	\draw (0.6, 1) node[circle, draw, fill=black!100, inner sep=0pt, minimum width=4pt] {} -- (0.6, 2) node[circle, draw, fill=black!100, inner sep=0pt, minimum width=4pt] {};	
	
	\draw (0, 0) node[circle, draw, fill=black!100, inner sep=0pt, minimum width=4pt] {} -- (-1.8, 1) node[circle, draw, fill=black!100, inner sep=0pt, minimum width=4pt] {};
	\draw (0, 0) node[circle, draw, fill=black!100, inner sep=0pt, minimum width=4pt] {} -- (1.8, 1) node[circle, draw, fill=black!100, inner sep=0pt, minimum width=4pt] {};
	
	\draw (0, -2) node[circle, draw, fill=black!100, inner sep=0pt, minimum width=4pt] {} -- (-1.8, -1) node[circle, draw, fill=black!100, inner sep=0pt, minimum width=4pt] {};
	\draw (0, -2) node[circle, draw, fill=black!100, inner sep=0pt, minimum width=4pt] {} -- (1.8, -1) node[circle, draw, fill=black!100, inner sep=0pt, minimum width=4pt] {};
	
	\node[label = {$\He^-(R) \cong \HFm(\Sigma(2, 7, 15))$}] at (0,2.5) {};

	\node[text width=0.1cm] at (6, -3.5) {\vdots};
	\draw (6, 0) node[circle, draw, fill=black!100, inner sep=0pt, minimum width=4pt] {} -- (6, -1) node[circle, draw, fill=black!100, inner sep=0pt, minimum width=4pt] {};
	\draw (6, -1) node[circle, draw, fill=black!100, inner sep=0pt, minimum width=4pt] {} -- (6, -2) node[circle, draw, fill=black!100, inner sep=0pt, minimum width=4pt] {};
	\draw (6, -2) node[circle, draw, fill=black!100, inner sep=0pt, minimum width=4pt] {} -- (6, -3) node[circle, draw, fill=black!100, inner sep=0pt, minimum width=4pt] {};
	
	\draw (6, 0) node[circle, draw, fill=black!100, inner sep=0pt, minimum width=4pt] {} -- (6-0.6, 1) node[circle, draw, fill=black!100, inner sep=0pt, minimum width=4pt] {};
	
	\draw (6-0.6, 1) node[circle, draw, fill=black!100, inner sep=0pt, minimum width=4pt] {} -- (6-0.6, 2) node[circle, draw, fill=black!100, inner sep=0pt, minimum width=4pt] {};
	
	\draw (6, 0) node[circle, draw, fill=black!100, inner sep=0pt, minimum width=4pt] {} -- (6-1.8, 1) node[circle, draw, fill=black!100, inner sep=0pt, minimum width=4pt] {};
	
	\draw (6, -2) node[circle, draw, fill=black!100, inner sep=0pt, minimum width=4pt] {} -- (6-1.8, -1) node[circle, draw, fill=black!100, inner sep=0pt, minimum width=4pt] {};
	
	\node[label = $\coker(1+J_0)$] at (5,2.5) {};

	\node[text width=0.1cm] at (9, -3.5+0.5) {\vdots};
	\draw (9, 0+0.5) node[circle, draw, fill=black!100, inner sep=0pt, minimum width=4pt] {} -- (9, -1+0.5) node[circle, draw, fill=black!100, inner sep=0pt, minimum width=4pt] {};
	\draw (9, -1+0.5) node[circle, draw, fill=black!100, inner sep=0pt, minimum width=4pt] {} -- (9, -2+0.5) node[circle, draw, fill=black!100, inner sep=0pt, minimum width=4pt] {};
	\draw (9, -2+0.5) node[circle, draw, fill=black!100, inner sep=0pt, minimum width=4pt] {} -- (9, -3+0.5) node[circle, draw, fill=black!100, inner sep=0pt, minimum width=4pt] {};
	
	\draw (9, 0+0.5) node[circle, draw, fill=black!100, inner sep=0pt, minimum width=4pt] {} ;
	\draw (9-0.6, 1+0.5) node[circle, draw, fill=black!100, inner sep=0pt, minimum width=4pt] {};
	\draw (9, 0+0.5) -- (9-0.6, 1+0.5) [dotted];
	
	\draw (9-0.6, 1+0.5) node[circle, draw, fill=black!100, inner sep=0pt, minimum width=4pt] {} -- (9-0.6, 2+0.5) node[circle, draw, fill=black!100, inner sep=0pt, minimum width=4pt] {};
	
	\draw (9, 0+0.5) node[circle, draw, fill=black!100, inner sep=0pt, minimum width=4pt] {}; 
	\draw (9-1.8, 1+0.5) node[circle, draw, fill=black!100, inner sep=0pt, minimum width=4pt] {};
	\draw (9, 0+0.5) --  (9-1.8, 1+0.5) [dotted];
	
	\draw (9, -2+0.5) node[circle, draw, fill=black!100, inner sep=0pt, minimum width=4pt] {}; 
	\draw (9-1.8, -1+0.5) node[circle, draw, fill=black!100, inner sep=0pt, minimum width=4pt] {};
	\draw (9, -2+0.5) -- (9-1.8, -1+0.5) [dotted];
	
	\node[label = $\ker(1+J_0)$] at (8,2.5) {};
	
	\draw (6, 0) -- (9, 0.5) [dashed];
	\draw (6, -1) -- (9, -1+0.5) [dashed];
	\draw (6, -2) -- (9, -2+0.5) [dashed];
	\draw (6, -3) -- (9, -3+0.5) [dashed];
	
\end{tikzpicture}
\caption{Graded root for $\HFIm(\Sigma(2, 7, 15))$. Dashed lines represent the action of $Q$; dotted lines are deleted edges in $\ker (1 + J_0)$. The half-root for $\ker (1 + J_0)$ has a grading shift of one relative to that for $\coker (1 + J_0)$.}
\label{fighfi}
\end{figure}

We now give an explicit description of the involutive Heegaard Floer correction terms for an AR plumbed three-manifold. First, recall the definition of the Neumann-Siebenmann invariant \cite{Neu}, \cite{Sieb}. Let $Y$ be a three-manifold with spin structure $\s$ and plumbing graph $G$, and let $L$ be the integer lattice spanned by the vertices of $G$. Amongst the characteristic vectors on $L$ corresponding to $\spinc$ structures limiting to $\s$ on $Y$, there is a unique vector $w$ (called the \textit{Wu vector}) whose coordinates in the natural basis of $G$ are all zero or one. Define the \textit{Neumann-Siebenmann invariant} of $(Y, \s)$ by
\[
\bar{\mu}(Y, \s) = \dfrac{1}{8}\left(\text{sign}(G) - w^2 \right),
\]
where sign$(G)$ is the signature of $G$ and $w^2$ is the self-pairing of $w$. In \cite{Neu} it is shown that $\bar{\mu}(Y, \s)$ is an integer lift of the Rokhlin invariant (defined for plumbed three-manifolds) and is independent of the choice of plumbing diagram. It was conjectured in \cite{Triangulations} that this invariant was equal to $\beta(-Y, \s)$ for all Seifert fibered rational homology spheres; this was proven in \cite{Dai} for the larger class of all AR plumbed three-manifolds. We obtain here the analogous result for the involutive Floer correction terms. 

\begin{proof}[Proof of Theorem~\ref{thm:ds}]
In order to compute $\du(Y, \s)$, we search for elements $x \in \HFIm(Y, \s)$ that have nonzero $U$-powers $U^nx$ for every $n$, and which eventually lie in the image of $Q$. Examining the description of $\HFIm(Y, \s)$ above, we see that any such $x$ must lie in the summand isomorphic to $\coker(1 + J_0)$. Moreover, it is clear that an $x$ with maximal grading satisfying these conditions is given by the element corresponding to the uppermost vertex in the graded root for $\coker(1 + J_0)$. The degree of this vertex is the same as the degree of the uppermost vertex in $R$, which implies that $\du(Y, \s) = d(Y, \s)$.

We now turn to $\underline{d}(Y, \s)$. In this case we search for $x$ whose $U$-powers do not lie in the image of $Q$; an examination of Figure \ref{fighfi} shows that such $x$ must lie in the copy of $\ker (1 + J_0)$. Since the graded root representing $\ker (1 + J_0)$ is not in general connected, the $x$ of maximal grading with the desired property is now represented by the uppermost vertex on the infinite stem of $\ker (1 + J_0)$. More precisely, we have that 
\[
\underline{d}(Y, \s) = g + 2, 
\]
where $g$ is the degree in $R_k[\sigma + 2]$ of the uppermost $J_0$-invariant vertex. In Theorem 3.2 of \cite{Dai}, the quantity $g+2$ was computed explicitly and shown to be equal to $- 2\bar{\mu}(Y, \s)$. \footnote{Recall from Section~\ref{sec:lattice} that reflecting the graded root $R_k[\sigma + 2]$ across the horizontal line of grading $-1$ yields a downwards opening graded root whose lattice homology computes the Heegaard Floer homology $\HFp(-Y, \s)$. Theorem 3.2 of \cite{Dai} states that the degree of the lowermost $J_0$-invariant vertex in this reflected root is $2\bar{\mu}(Y, \s)$.}
\end{proof}

\section{Monotone graded roots}\label{sec:monotone}

In this section, we construct a special class of graded roots that will be used in the sequel to carry out computations using the connected sum formula. The key result we prove here is that the involutive complex of any symmetric graded root is locally equivalent, in the sense of Definition~\ref{def:localE},  to the involutive complex of some root in this preferred class. For the purposes of computing $\du_{\Inv}$ and $\dl_{\Inv}$, it will thus suffice to work with these simpler roots, rather than the set of symmetric graded roots in general. Our result may also be thought of as providing a convenient set of representatives for AR manifolds in the group $\Inv_\Q$ (as defined in Section~\ref{sec:hfi}). 

Fix a positive integer $n$ and let $h_1, \ldots, h_n$ and $r_1, \ldots, r_n$ be sequences of rational numbers, all differing from one another by even integers, such that
\begin{enumerate}
\item $h_1 > h_2 > \cdots > h_n$, 
\item $r_1 < r_2 < \cdots < r_n$, and 
\item $h_n \geq r_n$.
\end{enumerate}
We construct a graded root $M = M(h_1, r_1; \ldots; h_n, r_n)$ associated to this data as follows. First, we form the stem of our graded root by drawing a single infinite $U$-tower with uppermost vertex in degree $r_n$. For each $1 \leq i < n$, we then introduce a symmetric pair of leaves $v_i$ and $J_0v_i$ in degree $h_i$ and connect these to the stem by using a pair of paths meeting the stem in degree $r_i$. (See Figure \ref{monotone}.) If $h_n > r_n$, we similarly introduce a pair of vertices $v_n$ and $J_0v_n$ in degree $h_n$ and connect these to the stem at $r_n$; in the ``degenerate" case that $h_n = r_n$, we declare $v_n$ to be the $J_0$-invariant vertex of degree $r_n$ lying on the stem and take no further action. Examples are given in Figure \ref{monotone}. 

\begin{figure}[h!]
\begin{tikzpicture}[thick,scale=0.6]%

	\node[text width=0.1cm] at (0, -4.5) {\vdots};
	
	\draw (0, 0) -- (0, -4);

	\draw (0, 0)  -- (0.6, 1);
	\draw (0, 0)  -- (-0.6, 1);
	
	\draw (0, -1) -- (1.8, 2);
	\draw (0, -1) -- (-1.8, 2);
	
	\draw (0, -3) -- (3.6, 3);
	\draw (0, -3) -- (-3.6, 3);

	\node[label = $v_1$] at (-3.6,2.8) {};
	\node[label = $J_0v_1$] at (3.6,2.8) {};
	
	\node[label = $v_2$] at (-1.8,1.8) {};
	\node[label = $J_0v_2$] at (1.8,1.8) {};
	
	\node[label = $v_3$] at (-0.6,0.8) {};
	\node[label = $J_0v_3$] at (0.6,0.8) {};

	\node[text width=0.1cm] at (8, -4.5) {\vdots};
	
	\draw (8, 1) -- (8, -4);
	
	\draw (8, -1) -- (8+1.8, 3);
	\draw (8, -1) -- (8-1.8, 3);
	
	\node[label = $v_1$] at (8-1.8,2.8) {};
	\node[label = $J_0v_1$] at (8+1.8,2.8) {};
	\node[label = $v_2$] at (8, 0.8) {};
	
	\node[label = \text{Degree}] at (-8, 3.4) {};
	\node[label = 4] at (-8, 2.4) {};
	\node[label = 2] at (-8, 1.4) {};
	\node[label = 0] at (-8, 0.4) {};
	\node[label = -2] at (-8, -0.6) {};
	\node[label = -4] at (-8, -1.6) {};
	\node[label = -6] at (-8, -2.6) {};
	\node[label = -8] at (-8, -3.6) {};
	
	\draw[loosely dotted] (-7, 3) -- (11, 3);
	\draw[loosely dotted] (-7, 2) -- (11, 2);
	\draw[loosely dotted] (-7, 1) -- (11, 1);
	\draw[loosely dotted] (-7, 0) -- (11, 0);
	\draw[loosely dotted] (-7, -1) -- (11, -1);
	\draw[loosely dotted] (-7, -2) -- (11, -2);
	\draw[loosely dotted] (-7, -3) -- (11, -3);
\end{tikzpicture}
\caption{$M(4,-8;2, -4; 0,-2)$ (left) and $M(4, -4; 0, 0)$ (right).}\label{monotone} 
\end{figure}
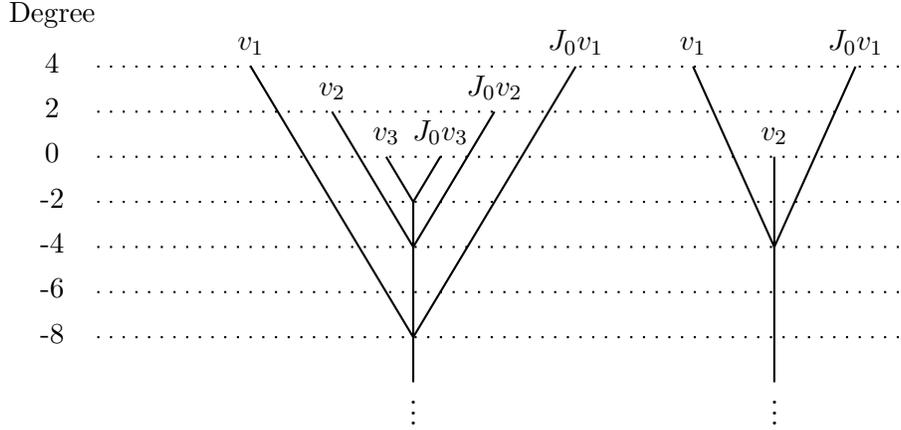

We call a symmetric graded root constructed in this way a \textit{monotone} graded root. Note that if $M$ is a monotone root, then the only vertices of $M$ which have valency greater than two lie on the stem and have degree equal to one of the $r_i$. Note also that due to the ordering of the $r_i$, our indexing of the leaves $v_i$ is consistent with the lexicographic left-to-right ordering described in Section \ref{sec:standard}. Geometrically, the ``monotonicity" condition says that the gradings of the leaves of $M$ form a strictly decreasing sequence as we move inwards towards the stem. 

We now investigate the standard complex of a monotone graded root parameterized by $M = M(h_1, r_1; \cdots; h_n, r_n)$. For simplicity, assume for the moment that $h_n$ is strictly greater than $r_n$. Following the prescription of Section \ref{sec:standard}, the standard complex $C_*(M)$ has $2n$ even generators corresponding to the leaves of $M$. Enumerating these by $v_1, \ldots, v_n$ and their reflections by $v_{n+1}=J_0v_n, \dots, v_{2n}=J_0v_1$, we see that $v_i$ has grading $h_i$ for $1 \leq i \leq n$ and $h_{2n+1-i}$ otherwise. Furthermore, $C_*(M)$ has $2n-1$ odd generators $\alpha_i$, which have gradings $r_i + 1$ for $1 \leq i \leq n$ and $r_{2n - i} + 1$ otherwise. The gradings of the $v_i$ are strictly decreasing as we move towards the stem, while the gradings of the $\alpha_i$ are strictly increasing. 

If $h_n = r_n$, then the prescription of Section \ref{sec:standard} gives us a complex with $2n - 1$ even generators. However, we can easily make a slight modification to $C_*(M)$ in this case by introducing a duplicate generator for $v_n$ in grading $r_n$ (in addition to the one which already exists) and adding an extra odd generator in grading $r_n + 1$ whose boundary is the sum of the two generators corresponding to $v_n$. This gives a complex with the same properties as those described in Section \ref{sec:standard}, but which is consistent with the notation of the previous paragraph. For convenience, we will thus use ``standard complex" to mean this slightly modified complex in the case that our graded root has an odd number of leaves. 

We now show that every symmetric graded root $R$ has a monotone subroot. This will be obtained as the smallest graded subroot of $R$ containing a certain subset of leaves. In general there may be multiple such subroots, but here we construct one which in some sense captures as much of $R$ as possible. 

We begin with some terminology. Let $v$ be any vertex in $R$. Denote by $\gamma_v$ the half-infinite path in $R$ starting at $v$ and running down the length of the stem. We define the \textit{base} $b(v)$ of $v$ to be the degree at which $\gamma_v$ first meets the stem; i.e., the maximal degree of any vertex in the intersection of the stem and $\gamma_v$. We denote the collection of all vertices in $R$ with a specified base $b$ by $C_b$, and call this set the \textit{cluster} of $R$ based at $b$. Geometrically, $C_b$ consists of all the branches of $R$ meeting the stem at degree $b$. 

In general, we will use the term ``cluster" to refer to a nontrivial cluster (i.e., one with more than one vertex). If $C_b$ is a nontrivial cluster, then it contains a pair of distinct leaves $v$ and $J_0v$ whose gradings are maximal amongst the vertices of $C_b$. We call these leaves the \textit{tip(s)} of $C_b$. If $C_b$ has more than one such pair, we arbitrarily select one to be the tip(s). Which one we select is not important, but we will use the fact that we have a specially marked pair of such vertices in each cluster. 

We now construct a subset $S$ of the leaves of $R$ using the following ``greedy" algorithm. Let $r$ be the degree of the uppermost $J_0$-invariant vertex $v$ in $R$. If the cluster $C_r$ is trivial, then we add $v$ to $S$; otherwise, we add the two tips of $C_r$ to $S$. Next, let $b$ be the greatest even integer strictly less than $r$ for which $C_b$ is nontrivial. If the tips of $C_b$ have degree strictly greater than all of the vertices currently in $S$, then we add them to $S$; otherwise, we do not. We continue on down the stem in this manner, adding tips of clusters precisely when they have height strictly greater than any of the vertices currently in $S$. We define $M$ to be the smallest graded subroot of $R$ that contains the leaves in $S$ after this process terminates. That is, identifying $R$ with its lattice homology $\He^-(R)$, $M$ is defined to be the $\ff[U]$-submodule spanned by the elements of $S$. This is clearly a monotone root; see Figure \ref{monotonesub} for an example. 

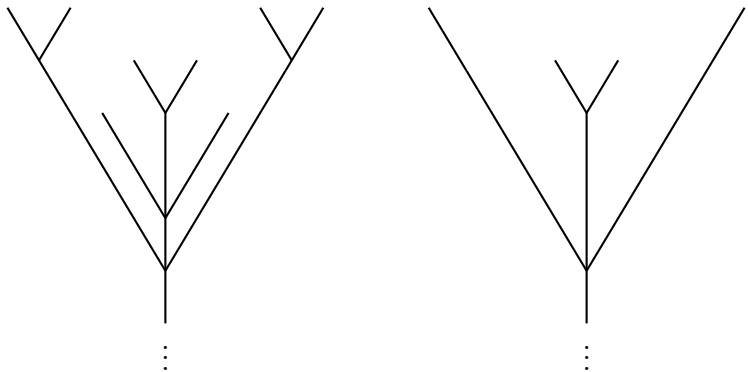
\begin{figure}[h!]
\begin{tikzpicture}[thick,scale=0.7]%

	\node[text width=0.1cm] at (0, -4.5) {\vdots};
	
	\draw (0, 0) -- (0, -4);

	\draw (0, 0)  -- (0.6, 1);
	\draw (0, 0)  -- (-0.6, 1);

	\draw (0, -2) -- (1.2, 0);
	\draw (0, -2) -- (-1.2, 0);

	\draw (0, -3) -- (3, 2);
	\draw (0, -3) -- (-3, 2);
	
	\draw (2.4, 1) -- (1.8, 2);
	\draw (-2.4, 1) -- (-1.8, 2);
	
	\node[text width=0.1cm] at (8, -4.5) {\vdots};
	
	\draw (8, 0) -- (8, -4);

	\draw (8, 0)  -- (8+0.6, 1);
	\draw (8, 0)  -- (8-0.6, 1);
	
	\draw (8, -3) -- (8+3, 2);
	\draw (8, -3) -- (8-3, 2);

\end{tikzpicture}
\caption{A graded root (left) and its monotone subroot (right).}\label{monotonesub} 
\end{figure}

The essential claim of this section is now:

\begin{theorem}\label{thm:monotonesubroot}
Let $R$ be a symmetric graded root. Then $(C_*(R), J_0)$ is locally equivalent to $(C_*(M), J_0)$, where $M$ is the monotone subroot of $R$ constructed above.
\end{theorem}
\begin{proof}
Let $R$ be a symmetric graded root with either $2n-1$ or $2n$ leaves, and select a pair of tips from each nontrivial cluster as described above. Before we begin the proof, it will be convenient to assume that $R$ has the following technical property with regard to the ordering of its leaves. Let $C_b$ be any nontrivial cluster in $R$. We require that out of all of the leaves of $R$ contained in $C_b$, the tips of $C_b$ are closest to the vertical line of symmetry in the lexicographic ordering (i.e., have indices closest to $n$). Such a property can easily be achieved by permuting the order of the leaves within each cluster in such a way that does not change the isomorphism class of $R$. If this property is satisfied, we say that $R$ is {\em nicely ordered}.

In order to provide some intuition for the proof, we first construct a pair of maps between $H_*(M)$ and $H_*(R)$ on the homological level. There is an obvious $\mathbb{F}[U]$-equivariant map $F : H_*(M) \rightarrow H_*(R)$ given by the inclusion of $M$ into $R$. To define a map $G: H_*(R) \rightarrow H_*(M)$, let $v$ be any vertex of $R$ lying to the left of the vertical line of symmetry. If $v$ already lies in $M$, then we define $G(v) = v$. Otherwise, we define $G(v)$ to be the closest vertex in $M$ with grading $\gr(v)$ that lies to the right of $v$. (Here we identify vertices with their corresponding representatives in the homology.) See Figure \ref{mapGex}. 

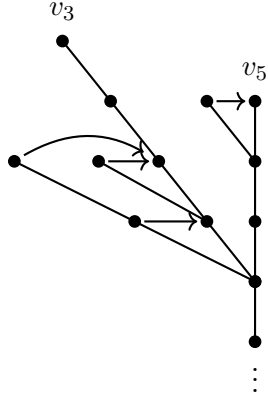
\begin{figure}[h!]
\begin{tikzpicture}[thick,scale=0.8]%

	\node[text width=0.1cm] at (0, -4.5) {\vdots};
	
	\node[label = $v_5$] at (0,0){};
	\node[circle, draw, fill=black!100, inner sep=0pt, minimum width=4pt] at (0,0) {};
	\node[circle, draw, fill=black!100, inner sep=0pt, minimum width=4pt] at (0,-1) {};
	\node[circle, draw, fill=black!100, inner sep=0pt, minimum width=4pt] at (0,-2) {};
	\node[circle, draw, fill=black!100, inner sep=0pt, minimum width=4pt] at (0,-3) {};
	\node[circle, draw, fill=black!100, inner sep=0pt, minimum width=4pt] at (0,-4) {};
		
	\draw (0,0) -- (0,-1);
	\draw (0,-1) -- (0,-2);
	\draw (0,-2) -- (0,-3);
	\draw (0,-3) -- (0,-4);

	\node[circle, draw, fill=black!100, inner sep=0pt, minimum width=4pt] at (-0.8,0) {};
	\draw (-0.8,0) -- (0,-1){};
	
	\node[label = $v_3$] at (-3.2,1){};
	\node[circle, draw, fill=black!100, inner sep=0pt, minimum width=4pt] at (-3.2,1) {};
	\node[circle, draw, fill=black!100, inner sep=0pt, minimum width=4pt] at (-2.4,0) {};
	\node[circle, draw, fill=black!100, inner sep=0pt, minimum width=4pt] at (-1.6,-1) {};
	\node[circle, draw, fill=black!100, inner sep=0pt, minimum width=4pt] at (-0.8,-2) {};
	\node[circle, draw, fill=black!100, inner sep=0pt, minimum width=4pt] at (0,-3) {};
	
	\draw (-3.2,1) -- (-2.4,0){};
	\draw (-2.4,0) -- (-1.6,-1){};
	\draw (-1.6,-1) -- (-0.8,-2){};
	\draw (-0.8,-2) -- (0,-3){};

	\node[circle, draw, fill=black!100, inner sep=0pt, minimum width=4pt] at (-2,-2) {};
	\node[circle, draw, fill=black!100, inner sep=0pt, minimum width=4pt] at (-4,-1) {};

	\node[circle, draw, fill=black!100, inner sep=0pt, minimum width=4pt] at (-2.6,-1) {};
	
	\draw (0,-3) -- (-2,-2){};
	\draw (-2,-2) -- (-4,-1){};
	\draw (-0.8,-2) -- (-2.6,-1){};
	
	\draw (-0.8+0.16,0)[->] -- (0-0.16,0){};
	\draw (-2+0.16,-2)[->] -- (-0.8-0.16,-2){};
	\draw [->] (-4+0.16,-1+0.1) to [out=30,in=150] (-1.6-0.25, -1+0.15);
	\draw (-2.6+0.16,-1)[->] -- (-1.6-0.16, -1){};

\end{tikzpicture}
\caption{The left half of a graded root $R$ and its monotone subroot $M$. In this example, $M$ is spanned by $v_3$ (and $J_0v_3$) and $v_5$. Arrows give the action of $G: H_*(R) \rightarrow H_*(M)$ on the vertices of $R$ not in $M$.}\label{mapGex} 
\end{figure}

This defines $G$ for vertices to the left of the stem; we extend $G$ to all the vertices of $R$ in the obvious $J_0$-equivariant way. The fact that $M$ is monotone (as well as the assumption that $R$ is nicely ordered) ensures that $G$ is equivariant with respect to the action of $U$. We think of $G$ as ``folding" $M$ into $R$ by collapsing all of the vertices of $R$ outside of $M$ towards the center line. 

We now show that $F$ and $G$ are induced by actual $J_0$-equivariant chain maps between $C_*(M)$ and $C_*(R)$. Let $v_1, \ldots, v_n$ be the generators of $C_*(R)$ lying to the left of the center line, and let $i_1 < i_2 < \cdots < i_k = n$ be the indices of the leaves in the left half of $R$ which lie in $S$. In order to remember the fact that $M$ is a subroot of $R$, it will be convenient to enumerate the even generators of $C_*(M)$ by $v_{i_1}, \ldots, v_{i_k}$ (and their reflections), rather than by using the indices $1$ through $k$. We denote the odd generators of $C_*(M)$ likewise by $\beta_{i_1}, \ldots, \beta_{i_{k-1}}$ (and their reflections), so that $\partial \beta_{i_j}$ is the sum of $U$-powers of $v_{i_j}$ and $v_{i_{j+1}}$. It is not hard to see that (as angles in $R$) $\beta_{i_j}$ and $\alpha_{i_j}$ are based at the same vertex and have the same sides. However, $\beta_{i_j}$ should not be identified with $\alpha_{i_j}$ as an element of $C_*(M)$, since $\partial \alpha_{i_j}$ is the sum of $U$-powers of $v_{i_j}$ and $v_{i_j+1}$. (The index here is $i_j + 1$, instead of $i_{j+1}$.) 

We now construct a chain map $f: C_*(M) \rightarrow C_*(R)$. On the even part of $C_*(M)$, we define this to be the inclusion sending $v_{i_j}$ as a generator of $C_*(M)$ to $v_{i_j}$ as a generator of $C_*(R)$. For $\beta_{i_j}$ an odd generator of $C_*(R)$, we define
\[
f(\beta_{i_j}) = \sum_{s = i_j}^{i_{j+1}-1} U^{(\gr(\alpha_s)-\gr(\beta_{i_j}))/2} \alpha_s.
\]
One can check that $\gr(\beta_{i_j}) = \gr(\alpha_{i_j}) \leq \gr(\alpha_s)$ for all $i_j \leq s < i_{j+1}$ due to the monotonicity of $M$ and the fact that $R$ is nicely ordered. This defines $f$ on generators of $C_*(R)$ to the left of the center line; we extend $f$ to generators on the right side of the center line in the obvious $J_0$-equivariant way. Then $f$ gives a $J_0$-equivariant chain map from $C_*(M)$ to $C_*(R)$ which induces the map $F$ on homology. The reader may find it helpful to consult Figure \ref{figchainmapex}.

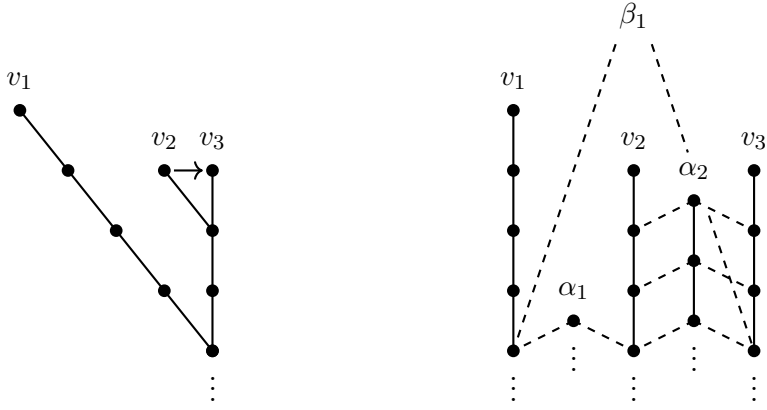
\begin{figure}[h!]
\begin{tikzpicture}[thick,scale=0.8]%

	\node[text width=0.1cm] at (0, -3.5) {\vdots};
	
	\node[label = $v_3$] at (0,0) {};
	\node[circle, draw, fill=black!100, inner sep=0pt, minimum width=4pt] at (0,0) {};
	\node[circle, draw, fill=black!100, inner sep=0pt, minimum width=4pt] at (0,-1) {};
	\node[circle, draw, fill=black!100, inner sep=0pt, minimum width=4pt] at (0,-2) {};
	\node[circle, draw, fill=black!100, inner sep=0pt, minimum width=4pt] at (0,-3) {};
		
	\draw (0,0) -- (0,-1);
	\draw (0,-1) -- (0,-2);
	\draw (0,-2) -- (0,-3);

	\node[label = $v_2$] at (-0.8,0) {};
	\node[circle, draw, fill=black!100, inner sep=0pt, minimum width=4pt] at (-0.8,0) {};
	\draw (-0.8,0) -- (0,-1){};
	
	\node[label = $v_1$] at (-3.2,1) {};
	\node[circle, draw, fill=black!100, inner sep=0pt, minimum width=4pt] at (-3.2,1) {};
	\node[circle, draw, fill=black!100, inner sep=0pt, minimum width=4pt] at (-2.4,0) {};
	\node[circle, draw, fill=black!100, inner sep=0pt, minimum width=4pt] at (-1.6,-1) {};
	\node[circle, draw, fill=black!100, inner sep=0pt, minimum width=4pt] at (-0.8,-2) {};
	\node[circle, draw, fill=black!100, inner sep=0pt, minimum width=4pt] at (0,-3) {};
	
	\draw (-3.2,1) -- (-2.4,0){};
	\draw (-2.4,0) -- (-1.6,-1){};
	\draw (-1.6,-1) -- (-0.8,-2){};
	\draw (-0.8,-2) -- (0,-3){};
	
	\draw (-0.8+0.16,0)[->] -- (0-0.16,0){};
	
	
	\node[label = $v_3$] at (9,0){};
	\node[circle, draw, fill=black!100, inner sep=0pt, minimum width=4pt] at (9,0) {};
	\node[circle, draw, fill=black!100, inner sep=0pt, minimum width=4pt] at (9,-1) {};
	\node[circle, draw, fill=black!100, inner sep=0pt, minimum width=4pt] at (9,-2) {};
	\node[circle, draw, fill=black!100, inner sep=0pt, minimum width=4pt] at (9,-3) {};
	\node[text width=0.1cm] at (9, -3.5) {\vdots};

	\node[label = $v_2$] at (7,0){};
	\node[circle, draw, fill=black!100, inner sep=0pt, minimum width=4pt] at (7,0) {};
	\node[circle, draw, fill=black!100, inner sep=0pt, minimum width=4pt] at (7,-1) {};
	\node[circle, draw, fill=black!100, inner sep=0pt, minimum width=4pt] at (7,-2) {};
	\node[circle, draw, fill=black!100, inner sep=0pt, minimum width=4pt] at (7,-3) {};
	\node[text width=0.1cm] at (7, -3.5) {\vdots};

	\node[label = $v_1$] at (5,1){};	
	\node[circle, draw, fill=black!100, inner sep=0pt, minimum width=4pt] at (5,1) {};
	\node[circle, draw, fill=black!100, inner sep=0pt, minimum width=4pt] at (5,0) {};
	\node[circle, draw, fill=black!100, inner sep=0pt, minimum width=4pt] at (5,-1) {};
	\node[circle, draw, fill=black!100, inner sep=0pt, minimum width=4pt] at (5,-2) {};
	\node[circle, draw, fill=black!100, inner sep=0pt, minimum width=4pt] at (5,-3) {};
	\node[text width=0.1cm] at (5, -3.5) {\vdots};	
	

	\node[label = $\alpha_2$] at (8,0-0.5) {};	
	\node[circle, draw, fill=black!100, inner sep=0pt, minimum width=4pt] at (8,0-0.5) {};
	\node[circle, draw, fill=black!100, inner sep=0pt, minimum width=4pt] at (8,-1-0.5) {};
	\node[circle, draw, fill=black!100, inner sep=0pt, minimum width=4pt] at (8,-2-0.5) {};
	\node[text width=0.1cm] at (8, -2.5-0.5) {\vdots};	

	\node[label = $\alpha_1$] at (6,-2-0.5){};		
	\node[circle, draw, fill=black!100, inner sep=0pt, minimum width=4pt] at (6,-2-0.5) {};
	\node[text width=0.1cm] at (6, -2.5-0.5) {\vdots};	
	
	\draw (9,0) -- (9, -3) {};
	\draw (7,0) -- (7, -3) {};
	\draw (5,1) -- (5, -3) {};
	\draw (8,0-0.5) -- (8, -2-0.5) {};
	
	\draw (8,0-0.5) [dashed] -- (9,-1) {};
	\draw (8,-1-0.5) [dashed] -- (9,-2) {};
	\draw (8,-2-0.5) [dashed] -- (9,-3) {};
	
	\draw (8,0-0.5) [dashed] -- (7,-1) {};
	\draw (8,-1-0.5) [dashed] -- (7,-2) {};
	\draw (8,-2-0.5) [dashed] -- (7,-3) {};
	
	\draw (6,-2-0.5) [dashed] -- (5,-3) {};
	\draw (6,-2-0.5) [dashed] -- (7,-3) {};

	\draw (7-0.3,3-0.9) [dashed] -- (5,-3) {};
	\draw (7+0.3,3-0.9) [dashed] -- (7.9,0.3) {};
	\draw (8.25,-0.75) [dashed] -- (9,-3) {};
	
	\node[label = $\beta_1$] at (7,2){};
	
\end{tikzpicture}
\caption{The left half of a graded root $R$ and its monotone subroot $M$, together with (half of) the complexes $C_*(R)$ and $C_*(M)$. In this example, $i_1 = 1$ and $i_2 = 3$, so $M$ is spanned by $v_1$ (and $J_0v_1$) and $v_3$. The generator $\beta_1$ of $C_*(M)$ has boundary $v_1 + U^3v_3$ and has grading equal to that of $\alpha_1$ (although not drawn that way).}
\label{figchainmapex}
\end{figure}

Finally, we construct a chain map $g: C_*(R) \rightarrow C_*(M)$. Let $v_s$ be an even generator of $C_*(R)$ lying to the left of the stem, and let $i_j$ be the least element of $i_1, \ldots, i_k$ greater than or equal to $s$. Again due to the structure of $M$ and $R$, it is not hard to see that $\gr(v_s) \leq \gr(v_{i_j})$. We thus define
\[
g(v_s) = U^{(\gr(v_{i_j})-\gr(v_s))/2} v_{i_j},
\]
extending $J_0$-equivariantly for generators to the right of the center line. Now let $\alpha_s$ be an odd generator of $C_*(R)$. If $s$ is an element of $i_1, \ldots, i_k$, we define $g(\alpha_s) = \beta_s$; otherwise, we define $g(\alpha_s) = 0$ (again extending to all of $C_*(R)$ $J_0$-equivariantly). We leave it to the reader to check that $g$ is a chain map which induces $G$. See Figure \ref{figchainmapex}. 

It is not hard to see that $F$ and $G$ induce the isomorphisms described above, between $H_*(M)$ and $H_*(R)$ after inverting the action of $U$. This completes the proof.
\end{proof}
 
If $R$ is a symmetric graded root, it turns out that the monotone root associated to $R$ is an invariant of $(C_*(R), J_0)$ up to local equivalence. Indeed, as a counterpart to Theorem~\ref{thm:monotonesubroot}, we have the following result:

\begin{theorem}
\label{thm:localeq}
Let $M$ and $M'$ be two monotone roots, with involutions $J_0$ and $J_0'$. If $(C_*(M), J_0)$ and $(C_*(M'), J_0')$ are locally equivalent $\inv$-complexes, then $M=M'$.
\end{theorem}

\begin{proof}
Let us write $M=M(h_1, r_1; \dots; h_n, r_n)$ and $M'=M(h_1', r_1'; \dots; h_m', r_m').$ Let $v_1, \dots, v_n$, $J_0v_n, \dots, J_0v_1$ be the leaves of $M$, and $v_1', \dots, v_m', J_0'v_m', \dots, J_0'v_1'$ the leaves of $M'$. We allow for the possibility that $M$ has an invariant leaf $v_n = J_0 v_n$. Similarly, $M'$ may have an invariant leaf $v_m' = J_0' v_m'$. 

A local equivalence between $(C_*(M), J_0)$ and $(C_*(M'), J_0')$ induces a grading-preserving map on homology
$$ F: H_*(M) \to H_*(M')$$
which commutes with the involutions and the $U$-actions and is an isomorphism in sufficiently low degrees. 

Given a leaf $v_i$ for $1\leq i\leq n$, write $F(v_i)$ as a linear combination of the leaves of $M'$, with the coefficients being powers of $U$. Note that all the terms in this sum (being in the same grading) become equal after multiplication with $U^k$ for $k \gg 0$. Since $F$ is an $\ff[U]$-module map, it follows from considering $F(U^k v_i)$ that $F(v_i)$ must contain an odd number of terms. Hence, our sum cannot consist only of paired terms $U^l v_j' + U^l J_0' v_j'$.  There is at least one value of $j \in \{1, \dots, m\}$ such that exactly one of $v_j'$ and $J_0'v_j'$ appears in $F(v_i)$ (with coefficient some power of $U$), or otherwise  $v_j'$ is invariant (that is, $j=m$ and $J_0 v_m' = v_m'$) and $U^l v_j'$ appears in $F(v_i)$. We denote by $f(i)=j$ the maximal $j$ with this property.

Without loss of generality, let us assume that $U^l v_j'$, rather than its conjugate, appears in $F(v_i)$. Since $F$ is grading-preserving, we then have that
$$h_i = h_j' - 2l \leq h_j'.$$ 
Moreover, the element $U^{(h_i - r_i)/2} v_i$ is $J_0$-invariant, and hence so is its image under $F$. Again by considering $F(U^kv_i)$ for $k \gg 0$, this implies that $F(U^{(h_i - r_i)/2} v_i)$ contains an element in the central stem of $M'$. Given how we chose $j$, that element must be a multiple of $v_j'$, namely $U^{l+(h_i - r_i)/2} v_j'$. Since the grading of this is equal to $r_i$, we have that $r_i \leq r_j'$.

We deduce that there exists a function $f: \{1, \dots, n\} \to\{1, \dots, m\}$ for which
$$ h_i \leq h_{f(i)}' \ \text{ and } \ r_i \leq r_{f(i)}' \ \ \text{for all } \ i=1, \dots, n.$$

On the other hand, the local equivalence also gives a map in the other direction, from $H_*(M')$ to $H_*(M)$, with similar properties. Thus, we get a function $g: \{1, \dots, m\} \to\{1, \dots, n\}$ such that
$$ h'_j \leq h_{g(j)} \ \text{ and } \ r_j' \leq r_{g(j)} \ \ \text{for all } \ j=1, \dots, m.$$

Considering the composition $g \circ f$, we obtain $h_i \leq h_{g(f(i))}$ for all $i$. Since $M$ is monotone, this means that $i \geq g(f(i))$ for all $i$. However, we also have that  $r_i \leq r_{g(f(i))}$, which implies the opposite inequality $i \leq g(f(i))$. The only possibility is that $g(f(i))=i$, so $g \circ f$ is the identity. 

By the same argument, the composition $f \circ g$ is the identity. We thus conclude that $f$ and $g$ are inverses to each other, and in particular that $n=m$, $h_i = h_i'$ and $r_i = r_i'$ for all $i$. This shows $M$ and $M'$ are identical.
\end{proof}


\section{Tensor products}
\label{sec:tensor}
In this section, we compute the involutive correction terms for the tensor product of any number of monotone roots. For such complexes, we will be able to express $\du_{\Inv}$ and $\dl_{\Inv}$ explicitly in terms of the parameters $(r_i, h_i)$ of each of the summands. Because of Theorem~\ref{thm:monotonesubroot}, these computations easily extend to the class of symmetric graded roots in general. 

Before we begin, it will be helpful to re-cast our definition of $C_*(R)$ in a geometric form similar to the construction of lattice homology given in Section~\ref{sec:lattice}. Let $R$ be a symmetric graded root and let $g$ be any even integer. For each even generator $U^nv_i$ of $C_*(R)$ lying in grading $g$, we draw a single 0-cell as in Figure \ref{1dsublevelset}, placing these on a horizontal line from left-to-right in the same order as they appear in $C_*(R)$. For each odd generator $U^n\alpha_i$ of $C_*(R)$ in grading $g+1$, we then draw a 1-cell connecting the two 0-cells corresponding to the terms in $\partial U^n\alpha_i$. See Figure \ref{1dsublevelset}. 
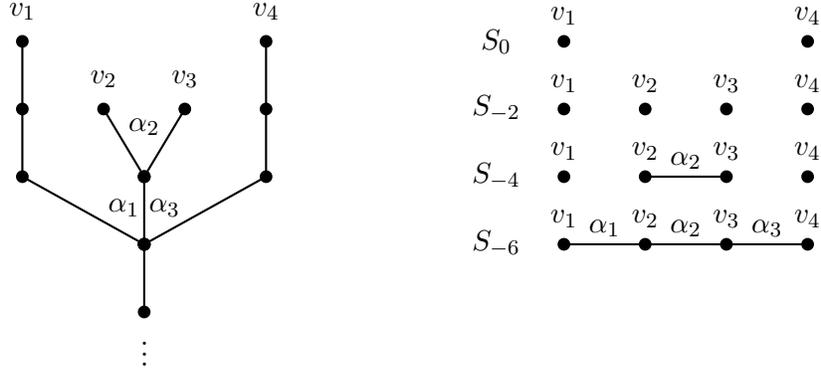
\begin{figure}[h!]
\begin{tikzpicture}[thick,scale=0.9]%

	\node[text width=0.1cm] at (0, -2.5) {\vdots};
	\draw (0, 0) node[circle, draw, fill=black!100, inner sep=0pt, minimum width=4pt] {} -- (0, -1) node[circle, draw, fill=black!100, inner sep=0pt, minimum width=4pt] {};
	\draw (0, -1) node[circle, draw, fill=black!100, inner sep=0pt, minimum width=4pt] {} -- (0, -2) node[circle, draw, fill=black!100, inner sep=0pt, minimum width=4pt] {};
	
	\draw (0, 0) node[circle, draw, fill=black!100, inner sep=0pt, minimum width=4pt] {} -- (-0.6, 1) node[circle, draw, fill=black!100, inner sep=0pt, minimum width=4pt] {};
	\draw (0, 0) node[circle, draw, fill=black!100, inner sep=0pt, minimum width=4pt] {} -- (0.6, 1) node[circle, draw, fill=black!100, inner sep=0pt, minimum width=4pt] {};
	
	\draw (0, -1) node[circle, draw, fill=black!100, inner sep=0pt, minimum width=4pt] {} -- (-1.8, 0) node[circle, draw, fill=black!100, inner sep=0pt, minimum width=4pt] {};
	\draw (0, -1) node[circle, draw, fill=black!100, inner sep=0pt, minimum width=4pt] {} -- (1.8, 0) node[circle, draw, fill=black!100, inner sep=0pt, minimum width=4pt] {};
	
	\draw (-1.8, 1) node[circle, draw, fill=black!100, inner sep=0pt, minimum width=4pt] {};
	\draw (1.8, 1) node[circle, draw, fill=black!100, inner sep=0pt, minimum width=4pt] {};
	\draw (-1.8, 2) node[circle, draw, fill=black!100, inner sep=0pt, minimum width=4pt] {};
	\draw (1.8, 2) node[circle, draw, fill=black!100, inner sep=0pt, minimum width=4pt] {};
	
	\draw (-1.8, 0) -- (-1.8, 1);
	\draw (-1.8, 1) -- (-1.8, 2);
	\draw (1.8, 0) -- (1.8, 1);
	\draw (1.8, 1) -- (1.8, 2);

	\node[label = $v_1$] at (-1.8,2) {};
	\node[label = $v_4$] at (1.8,2) {};
	\node[label = $v_2$] at (-0.6,1) {};
	\node[label = $v_3$] at (0.6,1) {};
	
	\node[label = $\alpha_2$] at (0,0.3) {};
	\node[label = $\alpha_1$] at (-0.3,-1+0.1) {};
	\node[label = $\alpha_3$] at (0.3,-1+0.1) {};

	\draw (-1.8+8, 2) node[circle, draw, fill=black!100, inner sep=0pt, minimum width=4pt, label = $v_1$] {};
	\draw (1.8+8, 2) node[circle, draw, fill=black!100, inner sep=0pt, minimum width=4pt, label = $v_4$] {};
	
	\draw (-1.8+8, 1) node[circle, draw, fill=black!100, inner sep=0pt, minimum width=4pt, label = $v_1$] {};
	\draw (-0.6+8, 1) node[circle, draw, fill=black!100, inner sep=0pt, minimum width=4pt, label = $v_2$] {};
	\draw (0.6+8, 1) node[circle, draw, fill=black!100, inner sep=0pt, minimum width=4pt, label = $v_3$] {};
	\draw (1.8+8, 1) node[circle, draw, fill=black!100, inner sep=0pt, minimum width=4pt, label = $v_4$] {};

	\draw (-1.8+8, 0) node[circle, draw, fill=black!100, inner sep=0pt, minimum width=4pt, label = $v_1$] {};
	\draw (-0.6+8, 0) node[circle, draw, fill=black!100, inner sep=0pt, minimum width=4pt, label = $v_2$] {};
	\draw (0.6+8, 0) node[circle, draw, fill=black!100, inner sep=0pt, minimum width=4pt, label = $v_3$] {};
	\draw (1.8+8, 0) node[circle, draw, fill=black!100, inner sep=0pt, minimum width=4pt, label = $v_4$] {};
	
	\draw (-1.8+8, -1) node[circle, draw, fill=black!100, inner sep=0pt, minimum width=4pt, label = $v_1$] {};
	\draw (-0.6+8, -1) node[circle, draw, fill=black!100, inner sep=0pt, minimum width=4pt, label = $v_2$] {};
	\draw (0.6+8, -1) node[circle, draw, fill=black!100, inner sep=0pt, minimum width=4pt, label = $v_3$] {};
	\draw (1.8+8, -1) node[circle, draw, fill=black!100, inner sep=0pt, minimum width=4pt, label = $v_4$] {};		
	
	\draw (-0.6+8, 0) -- (0.6+8, 0);
	\draw (-1.8+8, -1) -- (1.8+8, -1);
			
	\node[label = $S_0$] at (6-0.8,2-0.5) {};			
	\node[label = $S_{-2}$] at (6-0.8,1-0.5) {};	
	\node[label = $S_{-4}$] at (6-0.8,0-0.5) {};	
	\node[label = $S_{-6}$] at (6-0.8,-1-0.5) {};	
	
	\node[label = $\alpha_2$] at (8, 0-0.2) {};
	
	\node[label = $\alpha_2$] at (8, -1-0.2) {};
	\node[label = $\alpha_1$] at (8-1.2, -1-0.2) {};
	\node[label = $\alpha_3$] at (8+1.2, -1-0.2) {};
			
\end{tikzpicture}
\caption{Different sublevel sets for a graded root. The leaves $v_1$ and $v_4$ in this example have grading zero.}
\label{1dsublevelset}
\end{figure}

For convenience, we suppress writing the appropriate powers of $U$ and label these 0-cells and 1-cells with the vertex and angle labels $v_i$ and $\alpha_i$ of Section \ref{sec:standard}. Thus a 0-cell marked $v_i$ in grading $g$ means that (a) $\gr(v_i) \geq g$ and (b) the 0-cell in question represents the chain $U^{(\gr(v_i)-g)/2}v_i$. Similarly, a 1-cell marked $\alpha_i$ means that (a) $\gr(\alpha_i) \geq g + 1$ and (b) the 0-cell in question represents the chain $U^{(\gr(\alpha_i) - (g+1))/2}\alpha_i$.

By a slight abuse of notation, we call the (one-dimensional) cubical complex constructed above the \textit{sublevel set of weight g} and denote it by $S_g$. We give the cellular chain complex of this sublevel set a slightly modified grading by declaring a chain of dimension $i$ to have grading $g+ i$. It is clear that the chain complexes of these sublevel sets give a geometric model for $C_*(R)$. More precisely, let $C_*(S_g)$ be the graded chain complex
\[
C_*(S_g) = \bigoplus_i C_i(S_g, \mathbb{Z}/2\mathbb{Z}),
\]
where $C_i(S_g, \mathbb{Z}/2\mathbb{Z})$ is the usual $i$-dimensional cellular chain group of $S_g$ over $\mathbb{Z}/2\mathbb{Z}$ with the modified grading $g + i$. Then we have an isomorphism 
\[
C_*(R) \cong \bigoplus_{g \textit{ even}} C_*(S_g)
\]
of graded chain complexes over $\mathbb{F}[U]$, where the action of $U$ on the right is given by the induced inclusion map of $S_g$ into $S_{g-2}$. Passing to homology gives an isomorphism between $H_*(R)$ and the $\mathbb{Z}/2\mathbb{Z}$-homology of the sublevel sets. 

Now let $R_1, \ldots, R_k$ be a set of $k$ symmetric graded roots, and denote by $C_*(\Sigma)$ the tensor product complex $C_*(R_1) \otimes \cdots \otimes C_*(R_k)$. Let the dimension of any $x_1 \otimes \cdots \otimes x_k$ be given by the sum of the dimensions of each of the $x_i$. As before, for every even integer $g$ we construct a $k$-dimensional cubical complex $S_g$ by first drawing a 0-cell for each zero-dimensional generator of $C_*(\Sigma)$ in grading $g$. These are arranged in the obvious way as a subset of points in some $k$-dimensional product lattice (see Figure \ref{2dsublevelset}). For each $d$-dimensional generator of grading $g + d$ in $C_*(\Sigma)$, we likewise draw a $d$-dimensional lattice cube in such a way so that the cellular boundary coincides with the action of $\partial$ on $C_*(\Sigma)$. This gives a geometric model for $C_*(\Sigma)$. See Figure \ref{2dsublevelset}. 

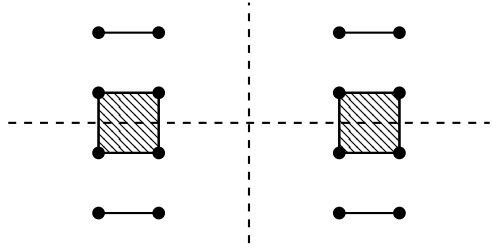
\begin{figure}[h!]
\begin{tikzpicture}[thick,scale=0.8]%

	\draw[dashed] (-4, 0) -- (4, 0);
	\draw[dashed] (0, -2) -- (0, 2);

	\draw (-2.5, 1.5) node[circle, draw, fill=black!100, inner sep=0pt, minimum width=4pt] {};
	\draw (-1.5, 1.5) node[circle, draw, fill=black!100, inner sep=0pt, minimum width=4pt] {};
	\draw (1.5, 1.5) node[circle, draw, fill=black!100, inner sep=0pt, minimum width=4pt] {};
	\draw (2.5, 1.5) node[circle, draw, fill=black!100, inner sep=0pt, minimum width=4pt] {};
		
	\draw (-2.5, 0.5) node[circle, draw, fill=black!100, inner sep=0pt, minimum width=4pt] {};
	\draw (-1.5, 0.5) node[circle, draw, fill=black!100, inner sep=0pt, minimum width=4pt] {};
	\draw (1.5, 0.5) node[circle, draw, fill=black!100, inner sep=0pt, minimum width=4pt] {};
	\draw (2.5, 0.5) node[circle, draw, fill=black!100, inner sep=0pt, minimum width=4pt] {};
	
	\draw (-2.5, -0.5) node[circle, draw, fill=black!100, inner sep=0pt, minimum width=4pt] {};
	\draw (-1.5, -0.5) node[circle, draw, fill=black!100, inner sep=0pt, minimum width=4pt] {};
	\draw (1.5, -0.5) node[circle, draw, fill=black!100, inner sep=0pt, minimum width=4pt] {};
	\draw (2.5, -0.5) node[circle, draw, fill=black!100, inner sep=0pt, minimum width=4pt] {};
	
	\draw (-2.5, -1.5) node[circle, draw, fill=black!100, inner sep=0pt, minimum width=4pt] {};
	\draw (-1.5, -1.5) node[circle, draw, fill=black!100, inner sep=0pt, minimum width=4pt] {};
	\draw (1.5, -1.5) node[circle, draw, fill=black!100, inner sep=0pt, minimum width=4pt] {};
	\draw (2.5, -1.5) node[circle, draw, fill=black!100, inner sep=0pt, minimum width=4pt] {};
	
	\draw (-2.5, 1.5) -- (-1.5, 1.5);
	\draw (-2.5, 0.5) -- (-1.5, 0.5);
	\draw (-2.5, -0.5) -- (-1.5, -0.5);
	\draw (-2.5, -1.5) -- (-1.5, -1.5);
	
	\draw (2.5, 1.5) -- (1.5, 1.5);
	\draw (2.5, 0.5) -- (1.5, 0.5);
	\draw (2.5, -0.5) -- (1.5, -0.5);
	\draw (2.5, -1.5) -- (1.5, -1.5);

	\draw (-2.5, 0.5) -- (-2.5, -0.5);
	\draw (-1.5, 0.5) -- (-1.5, -0.5);
	\draw (2.5, 0.5) -- (2.5, -0.5);
	\draw (1.5, 0.5) -- (1.5, -0.5);
	
	\draw[pattern=north west lines] (-2.5, 0.5) rectangle (-1.5,-0.5);
	\draw[pattern=north west lines] (1.5, 0.5) rectangle (2.5,-0.5);
\end{tikzpicture}
\caption{Example of a sublevel set for the connected sum of two roots ($k = 2$). Points (0-cells) represent generators $x = x_1 \otimes x_2$ where both $x_1$ and $x_2$ are leaf generators; edges (1-cells) represent generators where precisely one of $x_1$ and $x_2$ is an angle generator; squares (2-cells) represent generators where both of the $x_i$ are angle generators. Note the symmetry about each coordinate axis.}
\label{2dsublevelset}
\end{figure}

Note that every lattice cube may be labeled by $x = x_1 \otimes \cdots \otimes x_k$, where each $x_i$ is either a leaf or angle generator of $R_i$. As before, we suppress writing the necessary power of $U$, so that a lattice cube of dimension $d$ labeled by $x$ means that (a) $\gr(x) \geq g + d$ and (b) the cube in question represents the chain $U^{(\gr(x) - (g + d))/2}x$. If this holds, then we refer to $x$ as ``lying in the sublevel set $S_g$". Note also that $S_g$ is symmetric about reflection by $J_0$ in each coordinate. 

We now come to the following important lemma:
\begin{lemma} \label{lowgradinggenerators}
Let $R_1, \ldots, R_k$ be symmetric graded roots. Denote by $(C_*(\Sigma), J_0)$ the tensor product complex $C_*(R_1) \otimes \cdots \otimes C_*(R_k)$ with the reflection involution $J_0 \otimes \cdots \otimes J_0$. In sufficiently low gradings, the involutive homology $\HIm(\Sigma, J_0) := \HIm(C_*(\Sigma), J_0)$ consists of a single $U$-tower in even gradings and a single $U$-tower in odd gradings, with the action of $Q$ taking the odd tower onto the even tower. Then:
\begin{enumerate}
\item In sufficiently low even gradings, the single generator of $\HIm(\Sigma, J_0)$ is represented by $Qx$, where $x = x_1 + \cdots + x_{2n+1}$ is the sum of an odd number of points (0-cells) in $S_g$. Every representative is of this form and any two such choices are homologous.
\item In sufficiently low odd gradings, the single generator of $\HIm(\Sigma, J_0)$ is represented by a pair $(x, Qy)$, where $x$ is the sum of an odd number of 0-cells in $S_g$ and $y$ is some sum of 1-cells in $S_g$ such that $\partial y = x + J_0x$. Every representative is of this form and any two such choices are homologous.
\end{enumerate}
\end{lemma}
\begin{proof}
We first observe that $J_0$ preserves dimension. Thus we may filter the mapping cone of $(C_*(\Sigma), J_0)$ to obtain the double complex depicted in Figure~\ref{spectralseq}. Here, we have used $C_i$ to denote the $i$-dimensional generators (of any grading) in $C_*(\Sigma)$. Consider the spectral sequence from this double complex to the associated graded of $\HIm(\Sigma, J_0)$. In sufficiently low gradings, the sublevel sets $S_g$ are contractible, being solid $k$-dimensional boxes. In these gradings the spectral sequence collapses to give an $\mathbb{F}$-summand in place of $C_0$ and $QC_0$ and zeros elsewhere. This implies that in low gradings, $\HIm(\Sigma, J_0)$ is generated (in even gradings) by some element of $QC_0$ and (in odd gradings) by some sum of elements in $C_0$ and $QC_1$. Examining the definition of the mapping cone then proves the claim.
\end{proof}

\begin{figure}[h!]
\begin{tikzcd}
\hspace{2cm}\vdots \hspace{-0.4cm}\\
C_2 \arrow[r, "Q(1 + J_0)"] \arrow [d, "\partial"] & QC_2 \arrow [d, "\partial"] \\
C_1 \arrow[r, "Q(1 + J_0)"] \arrow [d, "\partial"] & QC_1 \arrow [d, "\partial"] \\
C_0 \arrow[r, "Q(1 + J_0)"] & QC_0
\end{tikzcd}
\caption{Mapping cone of $(C_*(\Sigma), J_0)$.}
\label{spectralseq}
\end{figure}

We now come to the computation of $\du_{\Inv}$ and $\dl_{\Inv}$ for connected sums. Here we will finally make use of the class of monotone roots constructed in the previous section. The statement of our theorem is a bit cumbersome, but the rationale behind various expressions will become clear after embarking on the proof.

\begin{theorem}\label{tensorproduct}
For $1 \leq i \leq k$, let $M_i$ be the monotone root
\[
M_i = M(h_1^i, r_1^i; \ldots; h_{n_i}^i, r_{n_i}^i).
\]
Denote by $\HIm(\Sigma, J_0)$ the involutive homology of the complex $C_*(M_1) \otimes \cdots \otimes C_*(M_k)$ with respect to the tensor product involution $J_0 \otimes \cdots \otimes J_0$. Let $\du_{\Inv}(\Sigma)$ and $\dl_{\Inv}(\Sigma)$ be the involutive correction terms of $\HIm(\Sigma, J_0)$. Then
\[
\du_{\Inv}(\Sigma) = \left(\sum_{i = 1}^k h_1^i\right) + 2.
\]
For each tuple of integers $(s_1, \ldots, s_k)$ with $1 \leq s_i \leq n_i$, define
\begin{align*}
B(s_1, \ldots, s_k) = \min\{ &r_{s_1}^1 + h_{s_2}^2 + h_{s_3}^3 + \cdots + h_{s_k}^k, \\
&h_{s_1}^1 + r_{s_2}^2 + h_{s_3}^3 + \cdots + h_{s_k}^k, \\
&h_{s_1}^1 + h_{s_2}^2 + r_{s_3}^3 + \cdots + h_{s_k}^k, \\
& \hspace{2cm} \vdots \\
&h_{s_1}^1 + h_{s_2}^2 + h_{s_3}^3 + \cdots + r_{s_k}^k\}.
\end{align*}
Then
\[
\dl_{\Inv}(\Sigma) = \max_{(s_1, \ldots, s_k)} B(s_1, \ldots, s_k) + 2,
\]
where the maximum is taken over all tuples $(s_1, \ldots, s_k)$ with $1 \leq s_i \leq n_i$.
\end{theorem}
\begin{proof}
We begin with $\du_{\Inv}$. According to its definition in Section~\ref{sec:hfi}, we search for elements in $\HIm(\Sigma, J_0)$ that have nonzero $U$-powers $U^n$ for every $n \geq 0$ and eventually lie in the image of $Q$. By Lemma \ref{lowgradinggenerators}, such an element must be represented by a chain $Qx$, where $x$ is the sum of an odd number of points in some sublevel set. Consider the single point $x = v_1^1 \otimes v_1^2 \otimes \cdots \otimes v_1^k$. Then $x$ has grading $h_1^1 + h_1^2 + \cdots + h_1^k$. There are no zero-dimensional generators of higher grading in $C_*(\Sigma)$, since each $v_1^i$ has maximal grading amongst the zero-dimensional generators of $C_*(M_i)$. This proves the claim for $\du_{\Inv}$. 

We now turn to $\dl_{\Inv}$. According to Lemma \ref{lowgradinggenerators}, we search for sublevel sets containing a pair $(x, Qy)$, where $x$ is the sum of an odd number 0-cells and $y$ is a sum of 1-cells for which $\partial y = x + J_0x$. Fix any tuple of integers $(s_1, \ldots, s_k)$ with $1 \leq s_i \leq n_i$. We begin by explicitly producing such a pair $(x, Qy)$ in the sublevel set of weight $B(s_1, \ldots, s_k)$. Since this element has grading $B(s_1, \ldots, s_k) + 1$ in $\HIm(\Sigma, J_0)$, if we can do this for every choice of $(s_1, \ldots, s_k)$, this will imply that $\dl_{\Inv}$ is bounded below by
\[
\dl_{\Inv}(\Sigma) \geq \max_{(s_1, \ldots, s_k)} B(s_1, \ldots, s_k) + 2.
\]
Let $x$ be the single point $x = v^1_{s_1} \otimes v^2_{s_2} \otimes \cdots \otimes v^k_{s_k}$. To see that this lies in $S_{B(s_1, \ldots, s_k)}$, observe that the grading of $x$ is given by $h^1_{s_1} + h^2_{s_2} + \cdots + h^k_{s_k}$. Since for each $M_i$, the $r$-parameters are all less than or equal to the $h$-parameters, we certainly have $\gr(x) \geq B(s_1, \ldots, s_k)$. Hence $x$ lies in $S_{B(s_1, \ldots, s_k)}$.

We now produce a path $\gamma$ in $S_{B(s_1, \ldots, s_k)}$ from $x$ to $J_0x$ which will serve as the desired $y$. (Here, by a \textit{path} we mean a sum of 1-cells which topologically constitutes a simple curve.) First consider the straight-line path $\gamma'$ from $x$ to the point
\[
x' = (J_0v^1_{s_1}) \otimes v^2_{s_2} \otimes \cdots \otimes v^k_{s_k}
\]
consisting of the edges 
\[
\sum_{i = s_1}^{2n_1-s_1-1} \alpha^1_{i} \otimes v^2_{s_2} \otimes \cdots \otimes v^k_{s_k}.
\]
We claim that this path lies in $S_{B(s_1, \ldots, s_k)}$. Indeed, the first and last edges in this path have grading
\[
(r_{s_1}^1 + 1) + h_{s_2}^2 + h_{s_3}^3 + \cdots + h_{s_k}^k
\] 
and thus lie in $S_{B(s_1, \ldots, s_k)}$. Because $M_1$ is monotone, the gradings of angle generators in $C_*(M_1)$ increase as we move towards the stem. Hence the intermediate edges in the above path all have gradings greater than that of the first and last, and thus also lie in $S_{B(s_1, \ldots, s_k)}$. 

Now let
\[
x^{(i)} = (J_0v^1_{s_1}) \otimes (J_0v^2_{s_2}) \otimes \cdots \otimes (J_0v^i_{s_i}) \otimes v^{i+1}_{s_{i+1}} \otimes v^{i+2}_{s_{i+2}} \otimes \cdots \otimes v^k_{s_k}.
\]
We similarly have straight-line paths from $x^{(i)}$ to $x^{(i+1)}$ lying in $S_{B(s_1, \ldots, s_k)}$ for each $i$. Concatenating these together gives the desired path $\gamma$ from $x$ to $J_0x$. Since this construction holds for any tuple $(s_1, \ldots, s_k)$, we have
\[
\dl_{\Inv}(\Sigma) \geq \max_{(s_1, \ldots, s_k)} B(s_1, \ldots, s_k) + 2,
\]
as claimed. See Figure \ref{rectpath} for a schematic picture of $\gamma$. 

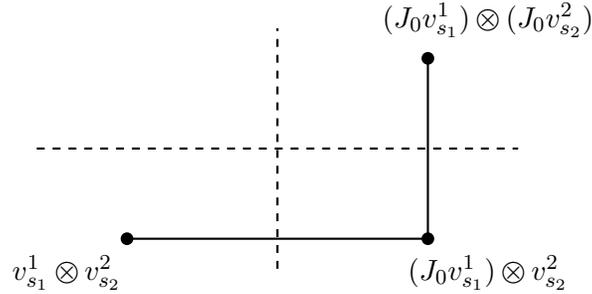
\begin{figure}[h!]
\begin{tikzpicture}[thick,scale=0.8]%

	\draw[dashed] (-4, 0) -- (4, 0);
	\draw[dashed] (0, -2) -- (0, 2);

	\draw (-2.5, -1.5) node[circle, draw, fill=black!100, inner sep=0pt, minimum width=4pt] {};
	\draw (2.5, 1.5) node[circle, draw, fill=black!100, inner sep=0pt, minimum width=4pt] {};
	
	\draw (2.5, -1.5) node[circle, draw, fill=black!100, inner sep=0pt, minimum width=4pt] {};
	
	\draw (-2.5, -1.5) -- (2.5, -1.5);
	\draw (2.5, -1.5) -- (2.5, 1.5);	
	
	\node[label = $v_{s_1}^1 \otimes v_{s_2}^2$] at (-2.5-1, -1.5-1.2) {};
	\node[label = $(J_0v_{s_1}^1) \otimes v_{s_2}^2$] at (2.5+1, -1.5-1.2) {};
	\node[label = $(J_0v_{s_1}^1) \otimes (J_0v_{s_2}^2)$] at (2.5+1, 1.5) {};	
\end{tikzpicture}
\caption{Example of $\gamma$.}
\label{rectpath}
\end{figure}
\noindent
The reason the above argument does not establish equality is that not every pair $(x, Qy)$ must consist of a symmetric pair of points together with a ``rectangular" path between them. Indeed, in general $x$ need not be a single point and $y$ need not be a simple curve. Even if this were the case, it is not clear that $y$ would have to be of the form presented in Figure \ref{rectpath}. We thus use a slightly different argument to prove the reverse inequality. 

Let  $(s_1, \ldots, s_k)$ be a tuple of integers with $1 \leq s_i \leq n_i$. We say that a point $v^1_{j_1} \otimes v^2_{j_2} \otimes \cdots \otimes v^k_{j_k}$ \textit{lies within the box} $I(s_1, \ldots, s_k)$ if $s_i \leq j_i \leq 2n_i + 1- s_i$ for each $i$. Geometrically, this means that each coordinate $v^i_{j_i}$ lies between $v^i_{s_i}$ and its reflection. We say that a path $\gamma$ is \textit{bounded by} $I(s_1, \ldots, s_k)$ if each vertex in the path is contained in $I(s_1, \ldots, s_k)$. See Figure \ref{boundedpath}. 

\begin{figure}[h!]
\begin{tikzpicture}[thick,scale=0.8]%

	\draw[dashed] (-4, 0) -- (4, 0);
	\draw[dashed] (0, -2) -- (0, 2);

	\draw (-1.5, 0.5) node[circle, draw, fill=black!100, inner sep=0pt, minimum width=4pt] {};
	\draw (1.5, -0.5) node[circle, draw, fill=black!100, inner sep=0pt, minimum width=4pt] {};
	
	\draw (-1.5, 0.5) -- (-0.5, 0.5);
	\draw (-0.5, 0.5) -- (-0.5, -1.5);
	\draw (-0.5, -1.5) -- (2.5, -1.5);
	\draw (2.5, -1.5) -- (2.5, -0.5);
	\draw (2.5, -0.5) -- (1.5, -0.5);
	
	\draw[dotted] (-2.5, -1.5) rectangle (2.5,1.5);
	
	\node[label = $v_{s_1}^1 \otimes v_{s_2}^2$] at (-2.5-1, -1.5-1.2) {};
	\node[label = $v_{s_1}^1 \otimes (J_0v_{s_2}^2)$] at (-2.5-1, 1.5) {};
	\node[label = $(J_0v_{s_1}^1) \otimes v_{s_2}^2$] at (2.5+1, -1.5-1.2) {};
	\node[label = $(J_0v_{s_1}^1) \otimes (J_0v_{s_2}^2)$] at (2.5+1, 1.5) {};			

\end{tikzpicture}
\caption{A path $\gamma$ bounded by $I(s_1, s_2)$.}
\label{boundedpath}
\end{figure}

Now suppose that $S_g$ is a sublevel set containing a pair $(x, Qy)$ for which $\partial y = x + J_0x$, as in Lemma \ref{lowgradinggenerators}. For the moment, let us assume that we are in the simplest case where $x$ consists of a single point and $y$ consists of a path $\gamma$ between $x$ and $J_0x$. Let $(s_1, \ldots, s_k)$ be the largest possible integers $1 \leq s_i \leq n_i$ for which $\gamma$ is bounded by $I(s_1, \ldots, s_k)$. Geometrically, this corresponds to finding the smallest possible symmetric ``box" into which $\gamma$ fits, as in Figure \ref{boundedpath}. We claim that $\gamma$ must either contain some edge of the form
\[
\alpha^1_{s_1} \otimes v^2_{j_2} \otimes v^3_{j_3} \otimes \cdots \otimes v^k_{j_k}
\]
or some edge of the form
\[
(J_0\alpha^1_{s_1}) \otimes v^2_{j_2} \otimes v^3_{j_3} \otimes \cdots \otimes v^k_{j_k}.
\]
This follows from the minimality of the box $I(s_1, \ldots, s_k)$ and the fact that $\gamma$ connects $x$ with $J_0x$. Indeed, $\gamma$ must either contain a point of the form 
\[
v^1_{s_1} \otimes v^2_{j_2} \otimes v^3_{j_3} \otimes \cdots \otimes v^k_{j_k}
\]
or a point of the form
\[
(J_0v^1_{s_1}) \otimes v^2_{j_2} \otimes v^3_{j_3} \otimes \cdots \otimes v^k_{j_k}
\]
by the fact that $s_1$ is maximal. On the other hand, the first coordinates of points in $\gamma$ cannot all be either $v^1_{s_1}$ or $J_0v^1_{s_1}$, since $\gamma$ connects $x$ with $J_0x$ and the first coordinates of these two points differ. This implies that $\gamma$ contains an edge of the form described above. 

Now observe that the grading of this edge is given by
\[
(r_{s_1}^1 + 1) + h_{j_2}^2 + h_{j_3}^3 + \cdots + h_{j_k}^k.
\]
Since all of the $R_i$ are monotone and $\gamma$ stays entirely within $I(s_1, \ldots, s_k)$, we additionally have that $h^i_{j_i} \leq h^i_{s_i}$ for each $i$. Hence the above expression is less than or equal to
\[
(r_{s_1}^1 + 1) + h_{s_2}^2 + h_{s_3}^3 + \cdots + h_{s_k}^k.
\]
This bounds $g+1$ from above, since $S_g$ contains an edge if and only if the grading of that edge is greater than or equal to $g+1$. Repeating this argument for each of the $k-1$ other indices yields
\[
g + 1 \leq  B(s_1, \ldots, s_k) + 1.
\]

Now suppose that we are in the situation where $x$ is the sum of multiple 0-cells and $y$ is not necessarily a simple curve. We claim that without loss of generality, we may still assume that $S_g$ contains a path $\gamma$ between a symmetric pair of points, as discussed above. To see this, first observe that a straightforward parity argument shows $x + J_0x$ must be equal to the sum of an odd number of symmetric pairs of points. More precisely, some symmetric pairs of points might appear in the expansion of $x$ and thus cancel each other out in $x + J_0x$, but because $x$ is the sum of an odd number of points, an odd number of symmetric pairs in $x + J_0x$ must remain.

Now, if $x + J_0x$ is equal to the sum of a single symmetric pair of points, then clearly $y$ must contain a path between these even if $y$ itself is not a simple curve. Thus, suppose that this is not the case. Fix any 0-cell $p$ appearing in the expansion of $x + J_0x$. If $y$ contains a path $\gamma$ from $p$ to $J_0p$, then we are done. If not, then $y$ must contain a path $\gamma$ from $p$ to some other 0-cell $q \neq J_0p$ appearing in the expansion of $x + J_0x$. Set 
\[
x' = x + p + q
\]
and 
\[
y' = y + \gamma + J_0\gamma.
\]
Then $(x', Qy')$ is another pair in $S_g$ for which $\partial y' = x' + J_0x'$. However, $x' + J_0x'$ now contains two fewer symmetric pairs than $x + J_0x$, since both $p + J_0p$ and $q + J_0q$ appear in the expansion of $x + J_0x$. Repeating this procedure inductively reduces to the case when $x + J_0x$ is equal to the sum of a single symmetric pair of points and establishes the claim.

We have thus shown that if $S_g$ contains a pair $(x, Qy)$ as described in Lemma \ref{lowgradinggenerators}, then $g$ is bounded above by $B(s_1, \ldots, s_k) + 1$ for some $(s_1, \ldots, s_k)$. Hence certainly
\[
\dl_{\Inv}(\Sigma) \leq \max_{(s_1, \ldots, s_k)} B(s_1, \ldots, s_k) + 2,
\]
completing the proof.
\end{proof}


\section{Applications}
\label{sec:app}

\subsection{Involutive correction terms of connected sums} 
\label{sec:dssums}
We now re-phrase the results of the previous section in terms of connected sums of actual AR manifolds. Let $Y$ be a three-manifold given as the oriented boundary of an almost-rational plumbing, and let $\s = [k]$ be a self-conjugate $\spinc$-structure on $Y$. Denote by $R$ the (shifted) graded root $R_k[\sigma + 2]$. As in Section \ref{sec:monotone}, $R$ has a monotone subroot $M$ parameterized by some tuple of rational numbers $M = M(h_1, r_1; \ldots; h_n, r_n)$. These parameters are related to $d$ and $\bar{\mu}$ by
\[
\du(Y, s) = d(Y, \s) = h_1 + 2
\]
and
\[
\dl(Y, s) = -2\bar{\mu}(Y, \s) = r_n + 2,
\]
as can be seen from the proof of Theorem~\ref{thm:ds}. Note that if $M$ only has two parameters $h_1$ and $r_1$, then $M$ is evidently specified by $d$ and $\bar{\mu}$. 

\begin{remark}
Let $R$ be a symmetric graded root, and $R^r$ its vertical reflection, which is a downwards-opening graded root. Let $v$ be the vertex of minimal grading in $R^r$ which is invariant under $\inv$. In \cite[Section 5.2]{Stoffregen}, Stoffregen defined $R^r$ to be of \textit{projective type} if there exists a vertex $w$ of minimal grading and a path from $v$ to $w$ that is grading-decreasing at each step. It is not hard to see that this corresponds precisely to the case when $n = 1$ above.
\end{remark}

As described in the Introduction, it will be helpful for us to re-parameterize $M$ in general by defining the quantities 
\[
d_i(Y, \s) = h_i + 2
\] 
and
\[
2\bar{\mu}_i(Y, \s) = - r_i - 2,
\]
for $i = 1, \ldots, n$, so that $d = d_1$ and $\bar{\mu} = \bar{\mu}_n$. In addition, we define
\[
2\tdelta_i(Y, \s) = d_i(Y, \s) + 2\bar \mu_i(Y, \s) = h_i - r_i.
\]
This parameterization is not particularly special, but we re-phrase our results in this language to more closely correspond with e.g.\ Theorem 1.4 of \cite{Stoffregen2}. Note that, in view of Theorems~\ref{thm:monotonesubroot} and \ref{thm:localeq}, the local equivalence class of $(\CFm(Y, \s), \inv)$ is determined by the quantities $d_i$ and $\bar \mu_i$.

We now prove a few of the results advertised in the Introduction.

\begin{proof}[Proof of Theorem~\ref{thm:dssums}]
According to Theorem 1.1 of \cite{HMZ}, the involutive Floer complex of a connected sum is homotopy equivalent to the (grading shifted) tensor product of the involutive complexes of the individual summands. Moreover, according to Lemma 8.7 of \cite{HMZ}, the relation of local equivalence is preserved under tensor product. Putting these together with Theorem~\ref{thm:standard}, we thus have a local equivalence between the pairs
\[
(C_*(\Sigma)[-2(k-1)], J_0) \cong (\CFm(Y_1 \# \dots \# Y_k, \s_1 \# \dots \# \s_k), \inv),
\]
where $C_*(\Sigma)$ is the tensor product complex of Theorem \ref{tensorproduct}. The theorem immediately follows by taking into account this grading shift and substituting in our new parameters. Note that $d_1(Y_i, \s_i) = d(Y_i, \s_i)$.
\end{proof}

\begin{proof}[Proof of Corollary~\ref{cor:dssums}]
Fix any $N(s_1, \ldots, s_k)$ as in Theorem \ref{thm:dssums}. By replacing the maximum over $i$ in the original definition of $N(s_1, \ldots, s_k)$ with the particular value $i = k$, we obtain the inequality
\[
N(s_1, \ldots, s_k) \leq \left( \sum_{i = 1}^k d_{s_i}(Y_i, \s_i) \right)- \left( 2\tdelta_{s_k}(Y_k, \s_k) \right).
\]
The right-hand side is then equal to
\[
\left( \sum_{i = 1}^{k-1} d_{s_i}(Y_i, \s_i) \right)- 2\bar{\mu}_{s_k}(Y_k, \s_k),
\]
which in turn is less than or equal to
\[
\Bigl( \sum_{i=1}^{k-1} d(Y_i, \s_i)\Bigr) - 2\bar \mu(Y_k, \s_k).
\]
Where here we have used the fact that $d = d_1 \geq d_i$ and $\bar{\mu} = \bar{\mu}_n \leq \bar{\mu}_i$. 

Applying Theorem~\ref{thm:dssums}(b), we get
\begin{equation}
\label{eq:dlmu}
\dl(Y_1 \# \dots \# Y_k, \s_1 \# \dots \# \s_k) \leq\Bigl( \sum_{i=1}^{k-1} d(Y_i, \s_i)\Bigr) - 2\bar \mu(Y_k, \s_k),
\end{equation}
as desired.

Now suppose that all of the $(Y_i, \s_i)$ are of projective type. Then the only valid tuple $(s_1, \ldots, s_k)$ in Theorem \ref{thm:dssums} is $(1, 1, \ldots, 1)$. Moreover, for each $i$ we have $\tdelta_1(Y_i, \s_i) = \tdelta(Y_i, \s_i)$ since $\bar{\mu}_1(Y_i, \s_i) = \bar{\mu}(Y_i, \s_i)$. It follows that we get equality in \eqref{eq:dlmu}. 
\end{proof}

\begin{proof}[Proof of Corollary~\ref{cor:stab}]
We first observe that for self-connected sums, to compute $\dl$ in Theorem \ref{thm:dssums}(b) it suffices to consider the maximum of $N(s_1, \ldots, s_k)$ over homogeneous tuples with every coordinate equal. Indeed, fix any $(s_1, \ldots, s_k)$, and suppose that $s_j$ is the least element of $s_1, \ldots, s_k$. Then
\[
N(s_1, \ldots, s_k) \leq \left( \sum_{i = 1}^k d_{s_i}(Y, \s) \right)- \left( 2\tdelta_{s_j}(Y, \s) \right),
\]
where we have replaced the maximum over $i$ in the original definition of $N(s_1, \ldots, s_k)$ with the particular value $i = j$. This is less than or equal to
\[
k \cdot d_{s_j}(Y, \s) - 2\tdelta_{s_j}(Y, \s) = N(s_j, \ldots, s_j).
\]
Hence it suffices to maximize over homogenous tuples, proving the claim. 

Now suppose that $k$ is large. All of the $\tdelta_i$ are bounded in magnitude, while $d = d_1$ is greater than any of the other $d_i$. It is thus clear that as $k$ grows to infinity, the expression in the corollary will eventually stabilize to $k \cdot d(Y, s) - 2\tdelta_1(Y, s)$. This completes the proof.
\end{proof}

\subsection{Mixing orientations}
Given an $\inv$-complex $(C, \inv)$, there is a dual $\inv$-complex $(C^\vee, \inv^\vee)$; cf. \cite[Section 8.3]{HMZ}. When $(C, \inv)$ is the Heegaard Floer $\inv$-complex for some rational homology sphere $(Y, \s)$,  the results of \cite[Section 4.2]{HMinvolutive} show that its dual is the Heegaard Floer $\inv$-complex for the same manifold, with the opposite orientation.

Since Theorem~\ref{thm:standard} already gives models for the Heegaard Floer $\inv$-complexes associated to AR plumbed $3$-manifolds, by dualizing we get models for the Heegaard Floer $\inv$-complexes of those manifolds with the opposite orientation. From here, by tensoring and then taking the involutive homology, we can calculate $\HFIm$ for any connected sum of AR plumbed $3$-manifolds, where the summands may come with either orientation. We do not have closed formulas for $\dl$ and $\du$ in this case, but we can do {\em ad hoc} computations in any particular example.

\begin{proof}[Proof of Theorem~\ref{thm:AllDifferent}]
We will make use of the calculations in  \cite[Section 9.2]{HMZ}, for the manifold 
$$Y' =S^3_{-3}(T_{2,7}) \# S^3_{-3}(T_{2,7}) \# S^3_{5}(-T_{2,11})$$
This differs from $Y$ in that we take large surgeries on the respective knots, instead of $\pm 1$ surgeries. The end result of the calculations there was \cite[Proposition 1.5]{HMZ}, which says that
$$ \dl(Y') = -2, \quad \quad d(Y') = 0, \quad \quad \du(Y') = 2.$$

In our setting, note that $\Sigma(2,7,15) = S^3_{-1}(T_{2,7})$ is of projective type, with $d=0$ and $\bar \mu =2$ (cf. \cite[Section 5]{NemethiGRS} and \cite[Section 5.2]{Stoffregen}). Thus, for the purposes of computing involutive correction terms, we can replace $\CFm(\Sigma(2,7,15))$ with a locally equivalent $\inv$-complex, associated to the simple graded root that starts in degree $-2$ and splits in degree $-6$. A picture of this complex is shown in Figure~\ref{fig:AllDifferent}(a). Observe that this is the same as the complex used to compute $\CFm(S^3_{-3}(T_{2,7}))$ in \cite[Figure 14]{HMZ}, except shifted upward in grading by $1/2$.

Similarly, $\Sigma(2,11,23)= S^3_{-1}(T_{2,11})$ is of projective type, with $d=0$ and $\bar \mu = 3$. We replace $\CFm(\Sigma(2,11,23))$ with the locally equivalent $\inv$-complex shown in  Figure~\ref{fig:AllDifferent}(b). Furthermore, $\CFm(-\Sigma(2,11,23))=\CFm(S^3_1(-T_{2,11}))$ is the dual to $\CFm(\Sigma(2,11,23)$ in the sense of $\inv$-complexes; cf. \cite[Theorem 1.8 and Section 8.3]{HMZ}. When we take the dual of the complex in Figure~\ref{fig:AllDifferent}(b) we obtain the one in Figure~\ref{fig:AllDifferent}(c). This turns out to be the same as the complex used to compute $\CFm(S^3_5(-T_{2,11}))$ in \cite[Figure 20]{HMZ},  except shifted downward by $1$.

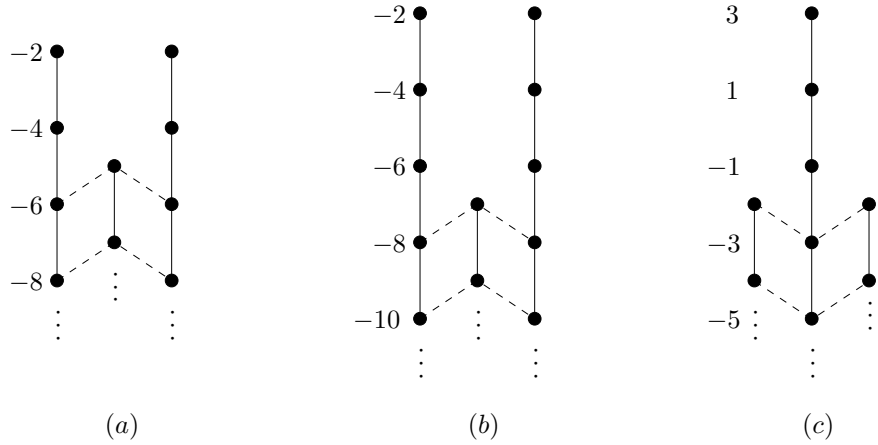
\begin {figure}
\begin {center}
\begin{picture}(0,0)%
\includegraphics{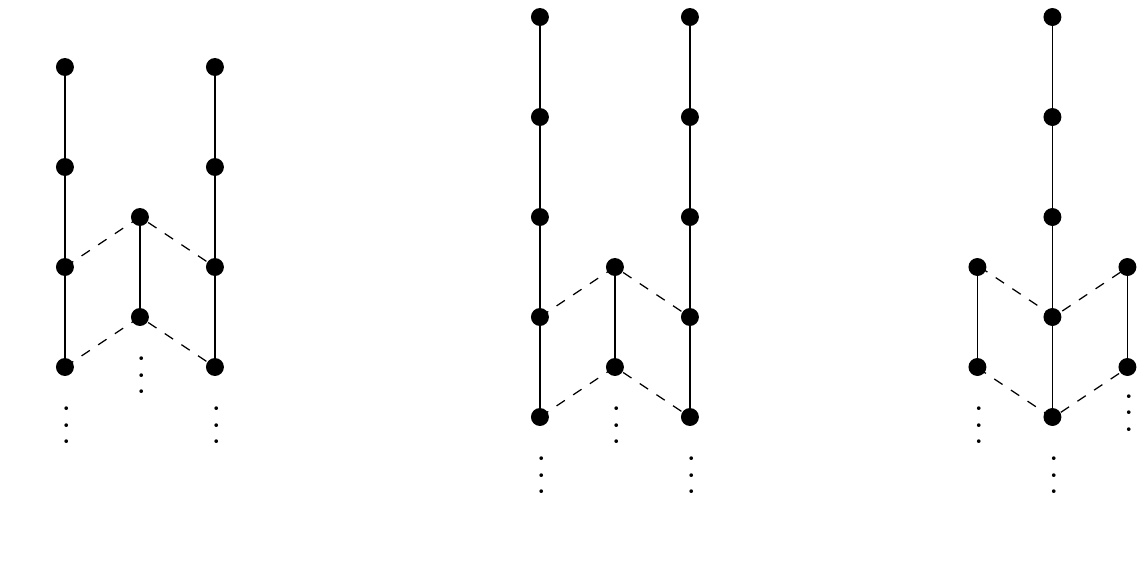}%
\end{picture}%
\setlength{\unitlength}{3158sp}%
\begingroup\makeatletter\ifx\SetFigFont\undefined%
\gdef\SetFigFont#1#2#3#4#5{%
  \reset@font\fontsize{#1}{#2pt}%
  \fontfamily{#3}\fontseries{#4}\fontshape{#5}%
  \selectfont}%
\fi\endgroup%
\begin{picture}(6823,3462)(811,-3739)
\put(6301,-1636){\makebox(0,0)[lb]{\smash{{\SetFigFont{10}{12.0}{\rmdefault}{\mddefault}{\updefault}{\color[rgb]{0,0,0}$-1$}%
}}}}
\put(6301,-2236){\makebox(0,0)[lb]{\smash{{\SetFigFont{10}{12.0}{\rmdefault}{\mddefault}{\updefault}{\color[rgb]{0,0,0}$-3$}%
}}}}
\put(6301,-2836){\makebox(0,0)[lb]{\smash{{\SetFigFont{10}{12.0}{\rmdefault}{\mddefault}{\updefault}{\color[rgb]{0,0,0}$-5$}%
}}}}
\put(3676,-436){\makebox(0,0)[lb]{\smash{{\SetFigFont{10}{12.0}{\rmdefault}{\mddefault}{\updefault}{\color[rgb]{0,0,0}$-2$}%
}}}}
\put(3676,-1036){\makebox(0,0)[lb]{\smash{{\SetFigFont{10}{12.0}{\rmdefault}{\mddefault}{\updefault}{\color[rgb]{0,0,0}$-4$}%
}}}}
\put(3676,-1636){\makebox(0,0)[lb]{\smash{{\SetFigFont{10}{12.0}{\rmdefault}{\mddefault}{\updefault}{\color[rgb]{0,0,0}$-6$}%
}}}}
\put(3676,-2236){\makebox(0,0)[lb]{\smash{{\SetFigFont{10}{12.0}{\rmdefault}{\mddefault}{\updefault}{\color[rgb]{0,0,0}$-8$}%
}}}}
\put(6451,-436){\makebox(0,0)[lb]{\smash{{\SetFigFont{10}{12.0}{\rmdefault}{\mddefault}{\updefault}{\color[rgb]{0,0,0}$3$}%
}}}}
\put(6451,-1036){\makebox(0,0)[lb]{\smash{{\SetFigFont{10}{12.0}{\rmdefault}{\mddefault}{\updefault}{\color[rgb]{0,0,0}$1$}%
}}}}
\put(3526,-2836){\makebox(0,0)[lb]{\smash{{\SetFigFont{10}{12.0}{\rmdefault}{\mddefault}{\updefault}{\color[rgb]{0,0,0}$-10$}%
}}}}
\put(826,-2536){\makebox(0,0)[lb]{\smash{{\SetFigFont{10}{12.0}{\rmdefault}{\mddefault}{\updefault}{\color[rgb]{0,0,0}$-8$}%
}}}}
\put(826,-1936){\makebox(0,0)[lb]{\smash{{\SetFigFont{10}{12.0}{\rmdefault}{\mddefault}{\updefault}{\color[rgb]{0,0,0}$-6$}%
}}}}
\put(826,-1336){\makebox(0,0)[lb]{\smash{{\SetFigFont{10}{12.0}{\rmdefault}{\mddefault}{\updefault}{\color[rgb]{0,0,0}$-4$}%
}}}}
\put(826,-736){\makebox(0,0)[lb]{\smash{{\SetFigFont{10}{12.0}{\rmdefault}{\mddefault}{\updefault}{\color[rgb]{0,0,0}$-2$}%
}}}}
\put(1576,-3661){\makebox(0,0)[lb]{\smash{{\SetFigFont{10}{12.0}{\rmdefault}{\mddefault}{\updefault}{\color[rgb]{0,0,0}$(a)$}%
}}}}
\put(4426,-3661){\makebox(0,0)[lb]{\smash{{\SetFigFont{10}{12.0}{\rmdefault}{\mddefault}{\updefault}{\color[rgb]{0,0,0}$(b)$}%
}}}}
\put(7051,-3661){\makebox(0,0)[lb]{\smash{{\SetFigFont{10}{12.0}{\rmdefault}{\mddefault}{\updefault}{\color[rgb]{0,0,0}$(c)$}%
}}}}
\end{picture}%

\caption {Complexes locally equivalent to: (a) $\CFm(\Sigma(2,7,15))$; (b) $\CFm(\Sigma(2,11,23))$; (c) $\CFm(-\Sigma(2,11,23))$. Solid lines represent the action of $U$, and dashed lines represent the action of $\del$, as usual. Homological degrees are indicated on the left side of each picture. }
\label{fig:AllDifferent}
\end {center}
\end {figure}

This implies that the calculations in  \cite[Section 9.2]{HMZ} work for $Y$ just as for $Y'$, with a total degree shift of $1/2 + 1/2 -1=0$. Hence, the correction terms $\dl, d, \du$ are the same for $Y$ as for $Y'$.
\end{proof}

\subsection{An infinite rank subgroup in homology cobordism}
Using the techniques developed in this paper, we now give a new proof that $\Theta^3_\Z$ has a $\Z^\infty$ subgroup.

\begin{proof}[Proof of Theorem~\ref{thm:infgen}]
 Consider the Brieskorn spheres
$$ Y_p =  \Sigma(p, 2p-1, 2p+1),\ p\geq 3, \ p \text{ odd}.$$

Stoffregen \cite[Theorem 8.9]{Stoffregen2} showed that $Y_p$ is of projective type, with $d(Y_p) = p-1$ and $\bar \mu(Y_p)=0.$

Using Theorem~\ref{thm:dssums} and Corollary~\ref{cor:dssums}, we can calculate the involutive correction terms of connected sums of $Y_p$'s. Specifically, if $p_1 \leq p_2 \leq \dots \leq p_k$, then
$$\du(Y_{p_1} \# \dots \# Y_{p_k}) = \sum_{i=1}^k d(Y_{p_i}), \ \ \
 \dl(Y_{p_1} \# \dots \# Y_{p_k})= \sum_{i=1}^{k-1} d(Y_{p_i}).$$
Hence,
\begin{equation}
\label{eq:yps}
(\du - \dl)(Y_{p_1} \# \dots \# Y_{p_k})= d(Y_{p_k}) = p_k -1.
\end{equation}

We claim that the classes $[Y_p] \in \Theta^3_\Z$ are linearly independent, so they generate a $\Z^\infty$ subgroup. Indeed, suppose we had a linear relation between the classes $[Y_p]$. By grouping the terms according to the sign of their coefficient, we can write the relation as
\begin{equation}
\label{eq:ypps}
 [Y_{p_1}] + [Y_{p_2}] +\dots + [Y_{p_k}] =    [Y_{p'_1}] + [Y_{p'_2}] +\dots + [Y_{p'_l}],
 \end{equation}
where 
$$ p_1 \leq p_2 \leq \dots \leq p_k, \ \ p'_1 \leq p'_2 \leq \dots \leq p'_l,$$
and
$$p_i \neq p'_j \text{ for any } i, j.$$ 
Since $\du$ and $\dl$ are homology cobordism invariants, so is their difference $\du - \dl$. By applying \eqref{eq:yps} to both sides of  \eqref{eq:ypps}, we find that $p_k = p'_l$, which is a contradiction.
\end{proof}

\bibliographystyle{custom} 
\bibliography{biblio} 

\end{document}